\documentclass[11pt,twoside, leqno]{article}

\usepackage{amssymb}
\usepackage{amsmath}
\usepackage{mathrsfs}
\usepackage{amsthm}

\allowdisplaybreaks

\def\ls{\lesssim}
\def\gs{\gtrsim}
\def\fz{\infty}

\renewcommand{\r}{\right}
\newcommand{\lf}{\left}

\def\ls{\lesssim}
\def\gs{\gtrsim}

\def\supp{{\mathop\mathrm{\,supp\,}}}

\def\aa{{\mathbb A}}
\def\rr{{\mathbb R}}
\def\rh{{\mathbb R}{\mathbb H}}
\def\rn{{{\rr}^n}}
\def\zz{{\mathbb Z}}
\def\nn{{\mathbb N}}
\def\cc{{\mathbb C}}

\newcommand{\wz}{\widetilde}
\newcommand{\oz}{\overline}

\newcommand{\ca}{{\mathcal A}}

\newcommand{\cd}{{\mathcal D}}

\newcommand{\cg}{{\mathcal G}}
\newcommand{\cm}{{\mathcal M}}
\newcommand{\cn}{{\mathcal N}}

\newcommand{\cs}{{\mathcal S}}

\def\az{\alpha}
\def\lz{\lambda}
\def\blz{\Lambda}

\def\bz{\beta}

\def\fai{\varphi}
\def\gz{{\gamma}}
\def\bgz{{\Gamma}}

\def\sz{\sigma}

\def\wz{\widetilde}

\def\ls{\lesssim}
\def\gs{\gtrsim}

\def\oz{\omega}

\def\uc{{\varepsilon}}
\def\gfz{\genfrac{}{}{0pt}{}}

\def\esup{\mathop\mathrm{\,ess\,sup\,}}

\def\hs{\hspace{0.3cm}}

\def\dint{\displaystyle\int}

\def\gfz{\genfrac{}{}{0pt}{}}

\newtheorem{theorem}{Theorem}[section]
\newtheorem{lemma}[theorem]{Lemma}
\newtheorem{corollary}[theorem]{Corollary}
\newtheorem{proposition}[theorem]{Proposition}
\theoremstyle{definition}
\newtheorem{remark}[theorem]{Remark}
\newtheorem{definition}[theorem]{Definition}
\def\supp{{\mathop\mathrm{\,supp\,}}}

\def\loc{{\mathop\mathrm{loc}}}

\numberwithin{equation}{section}

\begin{document}

\arraycolsep=1pt

\title{\Large\bf
Maximal Function Characterizations of Musielak-Orlicz-Hardy Spaces Associated to Non-negative Self-adjoint
Operators Satisfying Gaussian Estimates
\footnotetext{\hspace{-0.35cm} 2010 {\it Mathematics Subject
Classification}. {Primary 42B25; Secondary 42B35, 46E30.}
\endgraf{\it Key words and phrases}. Musielak-Orlicz-Hardy space,
atom, non-tangential maximal function, non-negative self-adjoint operator, Gaussian upper bound estimate.
\endgraf Dachun Yang is supported by the National
Natural Science Foundation  of China (Grant No. 11571039) and the Fundamental Research
Funds for Central Universities of China (Grant No. 2014KJJCA10).
Sibei Yang is supported by the
National Natural Science Foundation  of China (Grant Nos. 11401276 and 11571289).}}
\author{Dachun Yang and Sibei Yang\,\footnote{Corresponding author}}
\date{ }
\maketitle

\vspace{-0.6cm}

\begin{center}
\begin{minipage}{13.5cm}\small
{{\bf Abstract.} Let $L$ be a non-negative self-adjoint operator on $L^2(\mathbb{R}^n)$ whose heat kernels
have the Gaussian upper bound estimates.  Assume that the growth function $\varphi:\,\mathbb{R}^n\times[0,\infty)
\to[0,\infty)$ satisfies that $\varphi(x,\cdot)$ is an Orlicz function and
$\varphi(\cdot,t)\in {\mathbb A}_{\infty}(\mathbb{R}^n)$ (the class of
uniformly Muckenhoupt weights).  Let $H_{\varphi,\,L}(\mathbb{R}^n)$
be the Musielak-Orlicz-Hardy space introduced via the Lusin area
function associated with the heat semigroup of $L$.
In this article, the authors obtain several maximal function characterizations
of the space $H_{\varphi,\,L}(\mathbb{R}^n)$, which, especially,
answer an open question of L. Song and L. Yan
under an additional mild assumption satisfied by Schr\"odinger operators on $\mathbb{R}^n$
with non-negative potentials belonging to the reverse H\"older class, and
second-order divergence form elliptic operators on $\mathbb{R}^n$ with bounded
measurable real coefficients.}
\end{minipage}
\end{center}

\section{Introduction\label{s1}}

\hskip\parindent The real-variable theory of Hardy spaces on the $n$-dimensional
Euclidean space $\rn$, initiated by Stein and Weiss \cite{sw60} and then systematically developed by Fefferman
and Stein \cite{fs72}, plays an important role in various fields of analysis
(see, for example, \cite{fs72,st93}). It is well known that the Hardy space $H^p(\rn)$, with $p\in(0,1]$, is
a suitable substitute of the Lebesgue space $L^p(\rn)$; for example,
the classical Riesz transform is bounded on $H^p(\rn)$, but not on
$L^p(\rn)$ when $p\in(0,1]$. Moreover, the Hardy space $H^p (\rn)$ is essentially
related to the Laplace operator $\Delta:=\sum^n_{i=1}\frac {\partial^2}
{\partial x_i^2}$ on $\rn$. However, in many settings, these classical
function spaces are not applicable; for example, the Riesz transforms $\nabla L^{-1/2}$
may not be bounded from the Hardy space $H^1(\rn)$ to $L^1(\rn)$ when $L$ is a second-order
divergence form elliptic operator with complex bounded measurable coefficients (see \cite{hm09}).
Motivated by this, the study for the real-variable theory of various function
spaces associated with different differential operators has inspired great interests
in recent years (see, for example, \cite{adm,ar03,bckyy13b,dl13,hlmmy,hm09,jy10,jy11,jyy12,
ls13,sy10,sy15,y08,yy14,yys4}).

Let $L$ be a second-order divergence form elliptic operator on $\rn$ with bounded
measurable complex coefficients. The Hardy space $H^1_L(\rn)$ associated with $L$ was characterized
by Hofmann and Mayboroda \cite{hm09} in terms of the molecule, the Lusin area function, the non-tangential maximal
function ($\cn_L(f)$ or $\cn_{L,\,P}(f)$) or the radial
maximal function ($\mathcal{R}_L(f)$ or $\mathcal{R}_{L,\,P}(f)$), respectively,
associated with its heat semigroup or its Poisson semigroup
generated by $L$. Meanwhile,
the same equivalent characterizations of the Orlicz-Hardy space associated with $L$
as those of $H^1_L(\rn)$ were independently obtained in \cite{jy10}.
Recall that, for any $f\in L^2(\rn)$ and $x\in\rn$,
the \emph{non-tangential maximal function} $\cn_L(f)$ and the \emph{radial maximal function} $\mathcal{R}_L(f)$,
associated with the heat semigroup of $L$, are defined by
\begin{equation}\label{1.1}
\cn_L(f)(x):=\sup_{(y,t)\in\Gamma(x)}\lf\{\frac{1}{t^n}\int_{B(y,t)}\lf|e^{-t^2L}(f)(z)\r|^2\,dz\r\}^{1/2}
\end{equation}
and
\begin{equation}\label{1.2}
\mathcal{R}_L(f)(x):=\sup_{t\in(0,\fz)}\lf\{\frac{1}{t^n}\int_{B(x,t)}\lf|e^{-t^2L}(f)(z)\r|^2\,dz\r\}^{1/2},
\end{equation}
respectively, here and hereafter, for any $x\in\rn$, $\Gamma(x):=\{(y,t)\in\rn\times(0,\fz):\ |x-y|<t\}$
and, for any $(y,t)\in\rn\times(0,\fz)$, $B(y,t):=\{z\in\rn:\ |y-z|<t\}$.
Furthermore, the non-tangential maximal function $\cn_{L,\,P}(f)$ and the radial
maximal function $\mathcal{R}_{L,\,P}(f)$,
associated with the Poisson semigroup of $L$, are defined via replacing $e^{-t^2L}$
with $e^{-t\sqrt{L}}$ in \eqref{1.1} and \eqref{1.2}, respectively.

Moreover, let $L$ be a non-negative self-adjoint operator on $L^2(\rn)$ whose heat kernels
satisfy the Davies-Gaffney estimates. The equivalent characterizations of the Hardy space
$H^1_L(\rn)$ associated with $L$, including the atom, the molecule or the
Lusin area function associated with $L$, were established in \cite{hlmmy}, which were extended
to the Orlicz-Hardy space in \cite{jy11}. As a special case of this kind of operators,
when $L:=-\Delta+V$ is the Schr\"odinger operator with $0\le V\in L^1_{\loc}(\rn)$,
the non-tangential maximal function ($f^\ast_L$ or $f^\ast_{L,\,P}$)
or the radial maximal function ($f^+_L$ or $f^+_{L,\,P}$) characterizations,
associated with its heat semigroup or its Poisson semigroup,
of the Hardy space $H^1_L(\rn)$, the Orlicz-Hardy space
$H_{\Phi,\,L}(\rn)$ and the Musielak-Orlicz-Hardy space $H_{\fai,\,L}(\rn)$
were, respectively, obtained in \cite{hlmmy},
\cite{jy11} and \cite{bckyy13b,yys4}. Furthermore,
the same maximal function characterizations of the Hardy space $H^p_{L_A}(\rn)$,
with $p\in(0,1]$, and the Musielak-Orlicz-Hardy space
$H_{\fai,\,L_A}(\rn)$ associated with the magnetic Schr\"odinger operator
$L_A:=-(\nabla-iA)^2+V$ were, respectively, established in \cite{jyy12} and \cite{yy14},
where $A\in L^2_{\mathrm{loc}}(\mathbb{R}^n,\mathbb{R}^n)$
and $0\le V\in L^1_{\mathrm{loc}}(\mathbb{R}^n)$. Recall that,
for any $f\in L^2(\rn)$ and $x\in\rn$,
the \emph{non-tangential maximal function} $f^\ast_L$ and the \emph{radial maximal function} $f^+_L$,
associated with the heat semigroup of $L$, are defined by
\begin{equation}\label{1.3}
f^\ast_L(x):=\sup_{(y,t)\in\Gamma(x)}\lf|e^{-t^2L}(f)(y)\r|
\end{equation}
and
\begin{equation}\label{1.4}
f^+_L(x):=\sup_{t\in(0,\fz)}\lf|e^{-t^2L}(f)(x)\r|,
\end{equation}
respectively. Furthermore, the non-tangential maximal function $f^\ast_{L,\,P}$ and
the radial maximal function $f^+_{L,\,P}$,
associated with the Poisson semigroup of $L$, are defined via a similar way.
Observe that the maximal functions in \eqref{1.3} and \eqref{1.4} are different from those in
\eqref{1.1} and \eqref{1.2}. The main reason for
adding the averaging for the spatial variable in \eqref{1.1} and \eqref{1.2}
is that we need to compensate for the lack of pointwise estimates of
the heat semigroup and the Poisson semigroup in that case (see \cite{hm09} for more details).
Recall that, when $L:=-\Delta+V$ with $0\le V\in L^1_{\loc}(\rn)$, its heat semigroup and
its Poisson semigroup have pointwise estimates (see, for example, \cite[(8.4)]{hlmmy}).

From now on, assume that $L$ is a densely defined operator on $L^2(\rn)$ satisfying the
following two assumptions:
\begin{itemize}
  \item[{\bf(A1)}] $L$ is non-negative and self-adjoint;
  \item[{\bf(A2)}] the kernels of the semigroup $\{e^{-tL}\}_{t>0}$, denoted by $\{K_t\}_{t>0}$,
  are measurable functions on $\rn\times\rn$ and satisfy the Gaussian upper bound estimates, namely,
  there exist positive constants $C$ and $c$ such that, for all $t\in(0,\fz)$ and $x,\,y\in\rn$,
  \begin{equation}\label{1.5}
  \lf|K_t(x,y)\r|\le\frac{C}{t^{n/2}}\exp\lf\{-\frac{|x-y|^2}{ct}\r\}.
  \end{equation}
\end{itemize}
The typical examples of operators $L$, satisfying both the assumptions $(A1)$ and $(A2)$,
include the Schr\"odinger operator $L:=-\Delta+V$ with $0\le V\in L^1_{\loc}(\rn)$
and the second-order divergence form elliptic operator $L:=-\mathrm{div}(A\nabla)$
with $A:=\{a_{ij}\}_{i,\,j=1}^n$ satisfying
that, for any $i,\,j\in\{1,\,\ldots,\,n\}$, $a_{ij}$ is
a real measurable function on $\rn$ and there exists a constant $\lz\in(0,1]$ such that, for
all $i,\,j\in\{1,\,\ldots,\,n\}$ and $x,\,\xi\in\rn$,
$$a_{ij}(x)=a_{ji}(x)\ \ \text{and}\ \lz|\xi|^2\le\sum_{i,\,j=1}^n a_{ij}(x)
\xi_i\xi_j\le\lz^{-1}|\xi|^2.$$
Denote by $\cs(\rn)$ the \emph{space of all Schwartz functions} on $\rn$.
Let $p\in(0,1]$, $\az\in(0,\fz)$, $\phi\in\cs(\rr)$ be an even function and $\phi(0)=1$.
Recently, the characterizations of $H^p_L(\rn)$ in terms of the non-tangential maximal
function ($\phi^{\ast}_{L,\,\az}(f)$) or the grand maximal function ($\cg^{\ast}_L(f)$)
were obtained by Song and Yan \cite{sy15} via some essential improvements of techniques due to Calder\'on
\cite{c77}. Recall that, for any $f\in L^2(\rn)$ and $x\in\rn$,
the \emph{non-tangential maximal function} $\phi^{\ast}_{L,\,\az}(f)$ is defined by
\begin{equation}\label{1.6}
\phi^{\ast}_{L,\,\az}(f)(x):=\sup_{|y-x|<\az t,\,t\in(0,\fz)}\lf|\phi(t\sqrt{L})(f)(y)\r|
\end{equation}
(see \eqref{2.1} below for the definition of $\phi(t\sqrt{L})$) and
the \emph{grand maximal function} $\cg^{\ast}_L(f)$ is defined by
\begin{equation}\label{1.7}
\cg^{\ast}_L(f)(x):=\sup_{\phi\in\ca}\,\sup_{|x-y|<t,\,t\in(0,\fz)}\lf|\phi(t\sqrt{L})(f)(y)\r|,
\end{equation}
where
$$\ca:=\lf\{\phi\in\cs(\rr):\ \phi\ \text{is even with}\ \phi(0)\neq0,\ \int_\rr(1+|x|)^N\sum_{0\le k\le N}\lf|\frac{d^k\phi(x)}{dx^k}\r|^2\,dx\le1\r\}
$$
with $N$ being a large positive integer.
It is easy to see that, when $\phi(x):=e^{-x^2}$ for all $x\in\rr$ and $\az:=1$,
the maximal function $\phi^{\ast}_{L,\,\az}(f)$ in \eqref{1.6} coincides with
the maximal function $f^\ast_L$ in \eqref{1.3}.

Let the operator $L$ satisfy both the assumptions $(A1)$ and $(A2)$.
In this article, motivated by \cite{bckyy13b,sy15,yys4}, we
characterize the Musielak-Orlicz-Hardy space associated with $L$ via
the non-tangential maximal function in \eqref{1.6} or
the grand maximal function in \eqref{1.7}.
Under an additional assumption for $L$ (see
Assumption $(A3)$ below for the details), which is satisfied by Schr\"odinger operators on $\mathbb{R}^n$
with non-negative potentials belonging to the reverse H\"older class and
second-order divergence form elliptic operators on $\mathbb{R}^n$ with bounded
measurable real coefficients, we obtain the equivalent characterization of
the Musielak-Orlicz-Hardy space associated with $L$ in terms of
the radial maximal function in \eqref{1.4}. As a special case,
under the additional mild assumption $(A3)$ for $L$,
we give an answer for the open question in \cite[Remark 3.4]{sy15}
whether or not the Hardy space $H^p_L(\rn)$, with $p\in(0,1]$,
can be characterized via the radial maximal function in \eqref{1.4}.

To state the main results of this article, we now
describe the Musielak-Orlicz function considered in this article.
Recall that a function $\Phi:[0,\fz)\to[0,\fz)$ is called an \emph{Orlicz function}
if it is non-decreasing, $\Phi(0)=0$, $\Phi(t)>0$ for any $t\in(0,\fz)$ and
$\lim_{t\to\fz}\Phi(t)=\fz$ (see, for example,
\cite{m83,rr91}). We point out that, different from the classical definition of
Orlicz functions, the \emph{Orlicz functions in this article may not be convex}.
Moreover, $\Phi$ is said to be of \emph{upper} (resp. \emph{lower})
\emph{type $p$} for some $p\in(0,\fz)$ if
there exists a positive constant $C$ such that, for all
$s\in[1,\fz)$ (resp. $s\in[0,1]$) and $t\in[0,\fz)$,
$\Phi(st)\le Cs^p \Phi(t).$

For a given function $\fai:\,\rn\times[0,\fz)\to[0,\fz)$ such that, for
any $x\in\rn$, $\fai(x,\cdot)$ is an Orlicz function,
$\fai$ is said to be of \emph{uniformly upper} (resp.
\emph{lower}) \emph{type $p$}  for some $p\in(0,\fz)$ if there
exists a positive constant $C$ such that, for all $x\in\rn$,
$t\in[0,\fz)$ and $s\in[1,\fz)$ (resp. $s\in[0,1]$),
$\fai(x,st)\le Cs^p\fai(x,t)$.
Let
\begin{equation}\label{1.8}
I(\fai):=\inf\lf\{p\in(0,\fz):\ \fai\ \text{is of uniformly upper
type}\ p\r\}
\end{equation}
and
\begin{equation}\label{1.9}
i(\fai):=\sup\lf\{p\in(0,\fz):\ \fai\ \text{is of uniformly lower
type}\ p\r\}.
\end{equation}
In what follows, $I(\fai)$ and $i(\fai)$ are, respectively,
called the \emph{uniformly critical
upper type index} and the \emph{uniformly critical lower type index} of $\fai$.
Observe that $I(\fai)$ and $i(\fai)$ may not be attainable, namely, $\fai$ may not
be of uniformly upper type $I(\fai)$ or uniformly lower type $i(\fai)$
(see, for example, \cite{bckyy13b,hyy,k,yys4} for some examples).
Moreover, it is easy to see that, if $\fai$ is of uniformly upper type $p_0\in(0,\fz)$
and lower type $p_1\in(0,\fz)$, then $p_0\ge p_1$ and hence $I(\fai)\ge i(\fai)$.

\begin{definition}\label{d1.1}
Let $\fai:\rn\times[0,\fz)\to[0,\fz)$ satisfy that
$\fai(\cdot,t)$ is measurable for all $t\in[0,\fz)$.
Then $\fai$ is said to satisfy the
\emph{uniformly Muckenhoupt condition for some $q\in[1,\fz)$},
denoted by $\fai\in\aa_q(\rn)$, if, when $q\in (1,\fz)$,
\begin{equation*}
\aa_q (\fai):=\sup_{t\in
(0,\fz)}\sup_{B\subset\rn}\frac{1}{|B|^q}\int_B
\fai(x,t)\,dx \lf\{\int_B
[\fai(y,t)]^{1-q}\,dy\r\}^{q-1}<\fz
\end{equation*}
or
\begin{equation*}
\aa_1 (\fai):=\sup_{t\in (0,\fz)}
\sup_{B\subset\rn}\frac{1}{|B|}\int_B \fai(x,t)\,dx
\lf\{\esup_{y\in B}[\fai(y,t)]^{-1}\r\}<\fz,
\end{equation*}
where the first suprema are taken over all $t\in(0,\fz)$ and the
second ones over all balls $B\subset\rn$.

The function $\fai$ is said to satisfy the
\emph{uniformly reverse H\"older condition for some
$q\in(1,\fz]$}, denoted by $\fai\in \rh_q(\rn)$, if, when $q\in(1,\fz)$,
\begin{eqnarray*}
\rh_q (\fai):&&=\sup_{t\in(0,\fz)}\sup_{B\subset\rn}\lf\{\frac{1}
{|B|}\int_B [\fai(x,t)]^q\,dx\r\}^{1/q}\lf\{\frac{1}{|B|}\int_B
\fai(x,t)\,dx\r\}^{-1}<\fz
\end{eqnarray*}
or
\begin{equation*}
\rh_{\fz} (\fai):=\sup_{t\in(0,\fz)}\sup_{B\subset\rn}\lf\{\esup_{y\in
B}\fai(y,t)\r\}\lf\{\frac{1}{|B|}\int_B
\fai(x,t)\,dx\r\}^{-1} <\fz,
\end{equation*}
where the first suprema are taken over all $t\in(0,\fz)$ and the
second ones over all balls $B\subset\rn$.
\end{definition}

Recall that, in Definition \ref{d1.1},
$\aa_p(\rn)$,  with $p\in[1,\fz)$, and $\rh_q(\rn)$, with $q\in(1,\fz]$,
were respectively introduced in \cite{k} and \cite{yys4}.
Let $\aa_{\fz}(\rn):=\cup_{q\in[1,\fz)}\aa_{q}(\rn)$.
We now recall the notions of the \emph{critical indices} for $\fai\in\aa_{\fz}(\rn)$ as follows:
\begin{equation}\label{1.10}
q(\fai):=\inf\lf\{q\in[1,\fz):\ \fai\in\aa_{q}(\rn)\r\}
\end{equation}
and
\begin{equation}\label{1.11}
r(\fai):=\sup\lf\{q\in(1,\fz]:\ \fai\in\rh_{q}(\rn)\r\}.
\end{equation}
Recall also that, if $q(\fai)\in(1,\fz)$, then, by \cite[Lemma 2.4(iii)]{hyy},
we know that $\fai\not\in\aa_{q(\fai)}(\rn)$ and there exists
$\fai\not\in \aa_1(\rn)$ such that $q(\fai)=1$
(see, for example, \cite{jn87}).
Similarly, if $r(\fai)\in(1,\fz)$, then, by \cite[Lemma 2.3(iv)]{yy14},
we know that $\fai\not\in\rh_{r(\fai)}(\rn)$
and there exists $\fai\not\in\rh_\fz(\rn)$
such that $r(\fai)=\fz$ (see, for example, \cite{cn95}).

Now we recall the notion of growth functions from Ky \cite{k}.

\begin{definition}\label{d1.2}
A function $\fai:\rn\times[0,\fz)\rightarrow[0,\fz)$ is called
 a \emph{growth function} if the following hold true:
 \vspace{-0.25cm}
\begin{enumerate}
\item[(i)] $\fai$ is a \emph{Musielak-Orlicz function}, namely,
\vspace{-0.2cm}
\begin{enumerate}
    \item[(a)] $\fai(x,\cdot)$ is an
    Orlicz function for all $x\in\rn$;
    \vspace{-0.2cm}
    \item [(b)] $\fai(\cdot,t)$ is a measurable
    function for all $t\in[0,\fz)$.
\end{enumerate}
\vspace{-0.25cm} \item[(ii)] $\fai\in \aa_{\fz}(\rn)$.
\vspace{-0.25cm} \item[(iii)] The function $\fai$ is of
uniformly lower type $p$ for some $p\in(0,1]$ and of uniformly
upper type 1.
\end{enumerate}
\end{definition}

For a Musielak-Orlicz function $\fai$ as in Definition \ref{d1.1},
a measurable function $f$ on $\rn$ is said to be in the \emph{Musielak-Orlicz space}
$L^{\fai}(\rn)$ if $\int_{\rn}\fai(x,|f(x)|)\,dx<\fz$. Moreover,
for any $f\in L^{\fai}(\rn)$, the \emph{quasi-norm} of $f$ is defined by
$$\|f\|_{L^{\fai}(\rn)}:=\inf\lf\{\lz\in(0,\fz):\
\int_{\rn}\fai\lf(x,\frac{|f(x)|}{\lz}\r)\,dx\le1\r\}.$$

Clearly,
\begin{equation}\label{1.12}
\fai(x,t):=\oz(x)\Phi(t)
\end{equation} is a growth function if
$\oz\in A_{\fz}(\rn)$ and $\Phi$ is an Orlicz function of lower
type $p$ for some $p\in(0,1]$ and upper type 1.
Here and hereafter, $A_q(\rn)$ with $q\in[1,\fz]$ denotes the
\emph{class of Muckenhoupt weights}
(see, for example, \cite{gra1}).
A typical example of such functions $\Phi$ is $\Phi(t):=t^p$,
with $p\in(0,1]$, for all $t\in [0,\fz)$
(see, for example, \cite{hyy,k,k1,yys4} for more examples of such $\Phi$).
Another typical example of growth functions is given by
\begin{equation}\label{1.13}
\fai(x,t):=\frac{t}{\ln(e+|x|)+\ln(e+t)}
\end{equation}
for all $x\in\rn$ and $t\in[0,\fz)$; more precisely,
$\fai\in \aa_1(\rn)$, $\fai$ is of uniformly upper type 1 (indeed, $I(\fai)=1$,
which is attainable) and $i(\fai)=1$ which is not attainable (see \cite{k} for the details).

Now we recall the definition of the Musielak-Orlicz-Hardy space $H_{\fai,\,L}(\rn)$
introduced in \cite{bckyy13b,yys4}.

\begin{definition}\label{d1.3}
Let $L$ be an operator on $L^2(\rn)$ satisfying the assumptions $(A1)$
and $(A2)$, and $\fai$ as in Definition \ref{d1.2}.
For $f\in L^2 (\rn)$ and $x\in\rn$, the \emph{Lusin area
function, $S_{L}(f)$, associated with $L$} is defined by
\begin{equation*}
S_{L}(f)(x):=\lf\{\int_{\bgz(x)}\lf|t^2 Le^{-t^2L}(f)(y)\r|^2\frac{dy\,dt}{t^{n+1}}\r\}^{1/2}.
\end{equation*}

A function $f\in L^2 (\rn)$ is said to be in the set $\mathbb{H}_{\fai,\,L}(\rn)$ if $S_L(f)\in
L^{\fai}(\rn)$; moreover, define
$\|f\|_{H_{\fai,\,L}(\rn)}:=\|S_{L}(f)\|_{L^{\fai}(\rn)}$. Then the \emph{Musielak-Orlicz-Hardy space}
$H_{\fai,\,L}(\rn)$ is defined to be the completion of $\mathbb{H}_{\fai,\,L}(\rn)$ with respect to the
quasi-norm $\|\cdot\|_{H_{\fai,\,L}(\rn)}$.
\end{definition}

Moreover, we recall the following definitions of $(\fai,\,q,\,M)_L$-atoms and
atomic Musielak-Orlicz-Hardy spaces $H^{M,\,q}_{\fai,\,L,\,\mathrm{at}}(\rn)$ introduced in
\cite[Definitions 5.2 and 5.3]{bckyy13b}.

\begin{definition}\label{d1.4}
Let $L$ and $\fai$ be as in Definition \ref{d1.3}.
Assume that $q\in(1,\fz]$, $M\in\nn$ and $B\subset\rn$ is a ball.

{\rm(I)} Let $\cd(L^M)$ be the domain of $L^M$.
A function $\az\in L^q(\rn)$ is called a $(\fai,\,q,\,M)_L$-\emph{atom}
associated with the ball $B$ if there exists a function $b\in\cd(L^M)$
such that
\begin{itemize}
  \item[(i)] $\az=L^M b$;
  \item[(ii)] for all $j\in\{0,\,1,\,\ldots,\,M\}$, $\supp(L^j b)\subset B$;
  \item[(iii)] $\|(r^2_BL)^jb\|_{L^q(\rn)}\le r^{2M}_B|B|^{1/q}\|\chi_B\|^{-1}_{L^\fai(\rn)}$,
  where $r_B$ denotes the radius of $B$ and $j\in\{0,\,1,\,\ldots,\,M\}$.
\end{itemize}

{\rm(II)} For $f\in L^2(\rn)$,
\begin{eqnarray}\label{1.14}
f=\sum_j\lz_j\az_j
\end{eqnarray}
is called an \emph{atomic} $(\fai,\,q,\,M)_L$-\emph{representation}
of $f$ if, for any $j$, $\az_j$ is a
$(\fai,\,q,\,M)_L$-atom associated with the ball $B_j\subset\rn$,
the summation \eqref{1.14} converges in $L^2(\rn)$ and $\{\lz_j\}_j\subset\cc$ satisfies
that $\sum_j\fai(B_j,|\lz_j|\|\chi_{B_j}\|^{-1}_{L^\fai(\rn)})<\fz$.
Let
$$\mathbb{H}^{M,\,q}_{\fai,\,L,\,\mathrm{at}}(\rn):=\lf\{f\in L^2(\rn):\ f\ \text{has an atomic}\
(\fai,\,q,\,M)_L\text{-representation}\r\}
$$
equipped with the \emph{quasi-norm}
$$\|f\|_{H^{M,\,q}_{\fai,\,L,\,\mathrm{at}}(\rn)}:=\inf\lf\{
\blz\lf(\lf\{\lz_j\az_j\r\}_j\r):\ \sum_j\lz_j\az_j\
\text{is a}\ (\fai,\,q,\,M)_L\text{-representation of}\ f\r\},
$$
where the infimum is taken over all the atomic
$(\fai,\,q,\,M)_L$-representations of $f$ as above and
\begin{eqnarray*}
\blz\lf(\lf\{\lz_j\az_j\r\}_j\r):=\inf\lf\{\lz\in(0,\fz):\ \sum_j\fai\lf(B_j,\frac{|\lz_j|}
{\lz\|\chi_{B_j}\|_{L^\fai(\rn)}}\r)\le1\r\}.
\end{eqnarray*}

The \emph{atomic Musielak-Orlicz-Hardy space} $H^{M,\,q}_{\fai,\,L,\,\mathrm{at}}(\rn)$
is then defined as the completion of the set $\mathbb{H}^{M,\,q}_{\fai,\,L,\,\mathrm{at}}(\rn)$
with respect to the quasi-norm $\|\cdot\|_{H^{M,\,q}_{\fai,\,L,\,\mathrm{at}}(\rn)}$.
\end{definition}

Now we introduce Musielak-Orlicz-Hardy spaces via maximal functions associated with the operator $L$.

\begin{definition}\label{d1.5}
Let $L$ and $\fai$ be as in Definition \ref{d1.3}.

{\rm(i)} Assume that $\phi\in\cs(\rr)$ is an even function with $\phi(0)=1$ and
$\az\in(0,\fz)$. For any $f\in L^2(\rn)$, let $\phi^{\ast}_{L,\,\az}(f)$ be as in \eqref{1.6}.
A function $f\in L^2 (\rn)$ is said to be in the set $\mathbb{H}^{\phi,\,\az}_{\fai,\,L,\,\mathrm{max}}(\rn)$ if $\phi^{\ast}_{L,\,\az}(f)\in L^{\fai}(\rn)$; moreover, define
$\|f\|_{H^{\phi,\,\az}_{\fai,\,L,\,\mathrm{max}}(\rn)}:=\|\phi^{\ast}_{L,\,\az}(f)\|_{L^{\fai}(\rn)}$.
Then the \emph{Musielak-Orlicz-Hardy space} $H^{\phi,\,\az}_{\fai,\,L,\,\mathrm{max}}(\rn)$ is defined to
be the completion of $\mathbb{H}^{\phi,\,\az}_{\fai,\,L,\,\mathrm{max}}(\rn)$ with respect to the
quasi-norm $\|\cdot\|_{H^{\phi,\,\az}_{\fai,\,L,\,\mathrm{max}}(\rn)}$.

Specially, if $\phi(x):=e^{-x^2}$ for all $x\in\rr$ and $\az:=1$, denote $\phi^{\ast}_{L,\,\az}(f)$
simply by $f^\ast_L$ and, in this case, denote the space
$H^{\phi,\,\az}_{\fai,\,L,\,\mathrm{max}}(\rn)$ simply by $H_{\fai,\,L,\,\mathrm{max}}(\rn)$.

{\rm(ii)} For any $f\in L^2(\rn)$, let $\cg^{\ast}_L(f)$ be as in \eqref{1.7}.
The \emph{Musielak-Orlicz-Hardy space} $H^{\ca}_{\fai,\,L,\,\mathrm{max}}(\rn)$
is defined via replacing $\phi^{\ast}_{L,\,\az}(f)$ by $\cg^{\ast}_L(f)$ in the definition of
the space $H^{\phi,\,\az}_{\fai,\,L,\,\mathrm{max}}(\rn)$.
\end{definition}

Then the first main result of this article reads as follows.

\begin{theorem}\label{t1.1}
Let $L$ be an operator on $L^2(\rn)$ satisfying the assumptions $(A1)$ and $(A2)$,
and $\fai$ as in Definition \ref{d1.2}. Assume that $r(\fai)$, $I(\fai)$, $q(\fai)$
and $i(\fai)$ are, respectively, as in \eqref{1.11}, \eqref{1.8}, \eqref{1.10} and \eqref{1.9},
and $[r(\fai)]'$ denotes the conjugate exponent of $r(\fai)$.

{\rm(i)} For any $q\in([r(\fai)]'I(\fai),\fz]$,
$M\in\nn\cap(\frac{nq(\fai)}{2i(\fai)},\fz)$ and $\az\in(0,\fz)$,
the spaces $H^{M,\,q}_{\fai,\,L,\,\mathrm{at}}(\rn)$,
$H^{\phi,\,\az}_{\fai,\,L,\,\mathrm{max}}(\rn)$ and $H^{\ca}_{\fai,\,L,\,\mathrm{max}}(\rn)$
coincide with equivalent quasi-norms.

{\rm(ii)} For any $q\in([r(\fai)]'I(\fai),\fz)$,
$M\in\nn\cap(\frac{nq(\fai)}{2i(\fai)},\fz)$ and $\az\in(0,\fz)$,
the spaces $H_{\fai,\,L}(\rn)$, $H^{M,\,q}_{\fai,\,L,\,\mathrm{at}}(\rn)$,
$H^{\phi,\,\az}_{\fai,\,L,\,\mathrm{max}}(\rn)$ and $H^{\ca}_{\fai,\,L,\,\mathrm{max}}(\rn)$
coincide with equivalent quasi-norms.
\end{theorem}

The following chains of inequalities give the strategy of the proof
of Theorem \ref{t1.1}(i). For all $f\in H^{\phi,\,\az}_{\fai,\,L,\,\mathrm{max}}(\rn)\cap L^2(\rn)$, we have
\begin{eqnarray}\label{1.15}
\|f\|_{H^{\phi,\,\az}_{\fai,\,L,\,\mathrm{max}}(\rn)}&&\gs\|f\|_{H^{M,\,\fz}_{\fai,\,L,\,\mathrm{at}}(\rn)}\gs
\|f\|_{H^{M,\,q}_{\fai,\,L,\,\mathrm{at}}(\rn)}\\ \nonumber
&&\gs\|f\|_{H^{\ca}_{\fai,\,L,\,\mathrm{max}}(\rn)}\gs\|f\|_{H^{\phi,\,\az}_{\fai,\,L,\,\mathrm{max}}(\rn)},
\end{eqnarray}
where $q\in([r(\fai)]'I(\fai),\fz)$ and the implicit constants are independent of $f$.
We show the first inequality in \eqref{1.15} via borrowing some ideas from
the proof of \cite[Theorem 1.4]{sy15}. By the definitions of the spaces
$H^{M,\,q}_{\fai,\,L,\,\mathrm{at}}(\rn)$, $H^{\phi,\,\az}_{\fai,\,L,\,\mathrm{max}}(\rn)$
and $H^{\ca}_{\fai,\,L,\,\mathrm{max}}(\rn)$, we find that the second and  the fourth inequalities in \eqref{1.15} are obvious.
Moreover, we obtain the third inequality in \eqref{1.15} by establishing a pointwise estimate for $\az^\ast_L$,
where $\az$ is a $(\fai,\,q,\,M)_L$-atom (see \eqref{2.26} below for the details). Furthermore,
(ii) of Theorem \ref{t1.1} is obtained by (i) and the fact that the spaces $H_{\fai,\,L}(\rn)$ and $H^{M,\,q}_{\fai,\,L,\,\mathrm{at}}(\rn)$, with $q\in([r(\fai)]'I(\fai),\fz)$, coincide with equivalent quasi-norms,
which was established in \cite[Theorem 5.4]{bckyy13b}.

Let
\begin{eqnarray}\label{1.16}
L_A:=-(\nabla-iA)^2+V
\end{eqnarray}
be a magnetic Schr\"odinger operator  on $\rn$ with $n\ge3$,
where $A\in L^2_{\loc}(\rn,\rn)$ and the potential $V$ belongs to
the Kato class, namely,
$$\sup_{x\in\rn}\lim_{r\downarrow0}\int_{B(x,r)}\frac{|V(y)|}{|x-y|^{n-2}}\,dy=0.
$$
Moreover, the Kato norm of $V$ is defined by
$$\|V\|_{K}:=\sup_{x\in\rn}\int_{\rn}\frac{|V(y)|}{|x-y|^{n-2}}\,dy.
$$
For the potential $V$, let $V_+:=\max\{V,\,0\}$ and $V_-:=\min\{V,\,0\}$.
Under the assumption that $L_A$ is as in \eqref{1.16} with $V_+$ belonging to the Kato class
and $\|V_-\|_{K}<\pi^{n/2}/\Gamma(n/2-1)$, it was showed in \cite{cd12} that $L$
satisfies the assumptions $(A1)$ and $(A2)$. Thus, as a corollary of Theorem \ref{t1.1},
we have the following several equivalent characterizations of the Musielak-Orlicz-Hardy space $H_{\fai,\,L_A}(\rn)$
associated with $L_A$.

\begin{corollary}\label{c1.1}
Let $L_A$ be as in \eqref{1.16}, with $V_+$ belonging to the Kato class
and $\|V_-\|_{K}<\pi^{n/2}/\Gamma(n/2-1)$, and $\fai$ as in Definition \ref{d1.2}.

{\rm(i)} Assume that $q$, $M$ and $\az$ are as in Theorem \ref{t1.1}(i).
Then the spaces $H^{M,\,q}_{\fai,\,L_A,\,\mathrm{at}}(\rn)$,
$H^{\phi,\,\az}_{\fai,\,L_A,\,\mathrm{max}}(\rn)$ and $H^{\ca}_{\fai,\,L_A,\,\mathrm{max}}(\rn)$
coincide with equivalent quasi-norms.

{\rm(ii)} Assume that $q$, $M$ and $\az$ are as in Theorem \ref{t1.1}(ii).
Then the spaces $H_{\fai,\,L_A}(\rn)$, $H^{M,\,q}_{\fai,\,L_A,\,\mathrm{at}}(\rn)$,
$H^{\phi,\,\az}_{\fai,\,L_A,\,\mathrm{max}}(\rn)$ and $H^{\ca}_{\fai,\,L_A,\,\mathrm{max}}(\rn)$
coincide with equivalent quasi-norms.
\end{corollary}

\begin{remark}\label{r1.1}
We point out that the equivalences of $H^{M,\,\fz}_{\fai,\,L,\,\mathrm{at}}(\rn)$,
$H^{\phi,\,\az}_{\fai,\,L,\,\mathrm{max}}(\rn)$ and $H^{\ca}_{\fai,\,L,\,\mathrm{max}}(\rn)$
in Theorem \ref{t1.1}(i) or $H^{M,\,\fz}_{\fai,\,L_A,\,\mathrm{at}}(\rn)$ and
$H_{\fai,\,L_A,\,\mathrm{max}}(\rn)$ in Corollary \ref{c1.1}(i) are obtained in
\cite[Theorem 1.4, Corollary 3.2 and Proposition 3.3]{sy15} when $\fai(x,t):=t^p$, with $p\in(0,1]$,
for all $x\in\rn$ and $t\in[0,\fz)$. Moreover, Theorem \ref{t1.1} and Corollary \ref{c1.1}
are new even when $\fai$ is as in \eqref{1.12} or \eqref{1.13}.
\end{remark}

Let $L$ satisfy the assumption $(A2)$ and $\fai$ be as in Definition
\ref{d1.2}. For any $f\in L^2(\rn)$, let $f^{+}_L$ be as \eqref{1.4}.
Then the \emph{Musielak-Orlicz-Hardy space} $H_{\fai,\,L,\,\mathrm{rad}}(\rn)$
is defined via replacing $\phi^{\ast}_{L,\,\az}(f)$ by the radial maximal function $f^{+}_L$ in the definition of
the space $H^{\phi,\,\az}_{\fai,\,L,\,\mathrm{max}}(\rn)$.

By the definitions of the spaces $H_{\fai,\,L,\,\mathrm{max}}(\rn)$ and $H_{\fai,\,L,\,\mathrm{rad}}(\rn)$,
we know that the continuous inclusion of $H_{\fai,\,L,\,\mathrm{max}}(\rn)\subset H_{\fai,\,L,\,\mathrm{rad}}(\rn)$
holds true. It is a natural question whether or not the continuous inclusion of
$H_{\fai,\,L,\,\mathrm{rad}}(\rn)\subset H_{\fai,\,L,\,\mathrm{max}}(\rn)$ holds true.

To answer this question, we need to introduce another assumption for the operator $L$ as follows:
\begin{itemize}
  \item[{\bf(A3)}] There exist
positive constants $C$ and $\mu\in(0,1]$ such that, for all $t\in(0,\fz)$ and $x,\,y_1,\,y_2\in\rn$,
\begin{equation}\label{1.17}
|K_t (y_1,x)-K_t (y_2,x)|\le\frac{C}{t^{n/2}}\frac{|y_1 -y_2 |^{\mu}}{t^{\mu/2}}.
\end{equation}
\end{itemize}

We point out that there are lots of operators on $\rn$ satisfying the assumption $(A3)$;
for example, Schr\"odinger operators on $\rn$ with non-negative potentials
belonging to the reverse H\"older class (see, for example, \cite{dz02})
and second-order divergence form elliptic operators on $\mathbb{R}^n$
with bounded measurable real coefficients (see, for example, \cite{at98}).

\begin{theorem}\label{t1.2}
Let $L$ be an operator on $L^2(\rn)$ satisfying the assumptions $(A2)$ and $(A3)$,
and $\fai$ as in Definition \ref{d1.2}.
Then the spaces $H_{\fai,\,L,\,\mathrm{max}}(\rn)$ and $H_{\fai,\,L,\,\mathrm{rad}}(\rn)$
coincide with equivalent quasi-norms.
\end{theorem}

\begin{remark}\label{r1.2}
We show Theorem \ref{t1.2} via borrowing some ideas from the
proofs of \cite[Proposition 21]{ar03} and \cite[Lemma 5.3]{yys2}.
Under the additional assumption that $L$ satisfies $(A3)$, Theorem \ref{t1.2} gives an answer
to the open question stated in \cite[Remark 3.4]{sy15} by taking $\fai(x,t):=t^p$, with $p\in(0,1]$,
for all $x\in\rn$ and $t\in[0,\fz)$.
\end{remark}

As a corollary of Theorems \ref{t1.1} and \ref{t1.2}, we have the following conclusion.

\begin{corollary}\label{c1.2}
Let $L$ be an operator on $L^2(\rn)$ satisfying the assumptions $(A1)$, $(A2)$ and $(A3)$,
and $\fai$ as in Definition \ref{d1.2}. Assume that $q$, $M$ and $\az$ are as in Theorem \ref{t1.1}(ii).
Then the spaces $H_{\fai,\,L}(\rn)$, $H^{M,\,q}_{\fai,\,L,\,\mathrm{at}}(\rn)$,
$H^{\phi,\,\az}_{\fai,\,L,\,\mathrm{max}}(\rn)$, $H_{\fai,\,L,\,\mathrm{rad}}(\rn)$
and $H^{\ca}_{\fai,\,L,\,\mathrm{max}}(\rn)$ coincide with equivalent quasi-norms.
\end{corollary}

The layout of this article is as follows. Sections \ref{s2} and \ref{s3}
are, respectively, devoted to the proofs of Theorems \ref{t1.1} and \ref{t1.2}.

Finally we make some conventions on notation. Throughout the whole
article, we always denote by $C$ a \emph{positive constant} which is
independent of the main parameters, but it may vary from line to
line. We also use $C_{(\gz,\,\bz,\,\ldots)}$ to denote a  \emph{positive
constant} depending on the indicated parameters $\gz,$ $\bz$,
$\ldots$. The \emph{symbol} $A\ls B$ means that $A\le CB$. If $A\ls
B$ and $B\ls A$, then we write $A\sim B$.
For each ball $B:=B(x_B,r_B)\subset\rn$, with some $x_B\in\rn$ and
$r_B\in (0,\fz)$, and $\az\in (0,\fz)$, let $\az B:=B(x_B,\az r_B)$.
For any measurable subset $E$ of $\rn$, we denote by $\chi_{E}$ its
\emph{characteristic function}. We also let $\nn:=\{1,\, 2,\, \ldots\}$ and
$\zz_+:=\nn\cup\{0\}$. For any ball $B$ in $\rn$ and $j\in\zz_+$,
let $S_j(B):=(2^{j+1}B)\setminus(2^{j}B)$ with $j\in\nn$ and $S_0(B):=2B$.
Finally, for $q\in[1,\fz]$, we denote by $q'$ its \emph{conjugate exponent}, namely, $1/q + 1/q'= 1$.

\section{Proof of Theorem \ref{t1.1}\label{s2}}

\hskip\parindent In this section, we give out the proof of Theorem \ref{t1.1}. To this end,
we first recall some auxiliary conclusions.

For a non-negative self-adjoint operator $L$ on $L^2(\rn)$,
denote by $E_L$ the spectral measure associated with $L$.
Then, for any bounded Borel function $F:\,[0,\fz)\rightarrow\cc$,
the operator $F(L):\ L^2(\rn)\rightarrow L^2(\rn)$ is defined by the formula
\begin{eqnarray}\label{2.1}
F(L):=\int_0^\fz F(\lz)\,dE_L(\lz).
\end{eqnarray}

Then we have the following lemma, which was obtained in \cite[Lemma 3.5]{hlmmy}.

\begin{lemma}\label{l2.1}
Assume that the operator $L$ satisfies the assumptions $(A1)$ and $(A2)$.
Let $\phi\in C^\fz_c(\rr)$ be even and $\supp(\phi)\subset(-1,1)$. Denote by
$\Phi$ the Fourier transform of $\phi$. Then, for any $k\in\zz_+$,
the kernels $\{K_{(t^2L)^k\Phi(t\sqrt{L})}\}_{t>0}$ of the operators $\{(t^2L)^k\Phi(t\sqrt{L})\}_{t>0}$ satisfy
that there exists a positive constant $C$, depending on $n$, $k$ and $\Phi$, such that,
for all $t\in(0,\fz)$ and $x,\,y\in\rn$,
\begin{eqnarray*}
\supp\lf(K_{(t^2L)^k\Phi(t\sqrt{L})}\r)\subset\{(x,y)\in\rn\times\rn:\ |x-y|\le t\}
\end{eqnarray*}
and
\begin{eqnarray*}
\lf|K_{(t^2L)^k\Phi(t\sqrt{L})}(x,y)\r|\le Ct^{-n}.
\end{eqnarray*}
\end{lemma}

Denote by $\cm$ the \emph{Hardy-Littlewood maximal operator} on $\rn$, namely, for all $f\in L^1_{\loc}(\rn)$ and
$x\in\rn$,
$$\cm(f)(x):=\sup_{B\ni x}\frac{1}{|B|}\int_B|f(y)|\,dy,$$
where the supremum is taken over all balls $B\ni x$.

Moreover, we have the following properties of growth functions,
which were obtained in \cite[Lemmas 4.1 and 4.2]{k}.

\begin{lemma}\label{l2.2}
Let $\fai$ be as in Definition \ref{d1.2}.

{\rm(i)} There exists a positive constant $C$ such that, for all
$(x,t_j)\in\rn\times[0,\fz)$ with $j\in\nn$,
$\fai(x,\sum_{j=1}^{\fz}t_j)\le C\sum_{j=1}^{\fz}\fai(x,t_j)$.

{\rm(ii)} Let $\wz{\fai}(x,t):=\int_0^t\frac{\fai(x,s)}{s}\,ds$ for all
$(x,t)\in\rn\times[0,\fz)$. Then $\wz{\fai}$ is equivalent to $\fai$, namely,
there exists a positive constant $C$ such that, for all $(x,t)\in\rn\times[0,\fz)$,
$C^{-1}\fai(x,t)\le\wz{\fai}(x,t)\le C\fai(x,t)$.

{\rm(iii)} If $p\in(1,\fz)$ and $\fai\in \aa_{p}(\rn)$, then
there exists a positive constant $C$ such that, for all measurable
functions $f$ on $\rn$ and $t\in[0,\fz)$,
$$\int_{\rn}\lf[\cm(f)(x)\r]^p\fai(x,t)\,dx\le
C\int_{\rn}|f(x)|^p\fai(x,t)\,dx.$$

{\rm(iv)} If $\fai\in \aa_{p}(\rn)$ with $p\in[1,\fz)$, then
there exists a positive constant $C$ such that, for all balls
$B_1,\,B_2\subset\rn$ with $B_1\subset B_2$ and $t\in(0,\fz)$,
$\frac{\fai(B_2,t)}{\fai(B_1,t)}\le
C[\frac{|B_2|}{|B_1|}]^p.$
\end{lemma}

Moreover, to show Theorem \ref{t1.1}, we need to establish the following conclusion.

\begin{proposition}\label{p2.1}
Let $L$ satisfy the assumptions $(A1)$
and $(A2)$, and $\fai$ be as in Definition \ref{d1.2}.
Assume that $\psi_1,\,\psi_2\in\cs(\rr)$ are even functions with $\psi_1(0)=\psi_2(0)=1$,
and $\az_1,\,\az_2\in(0,\fz)$. For $i\in\{0,\,1\}$, $f\in L^2(\rn)$ and $x\in\rn$,
let
$$\psi^{\ast}_{i,\,L,\,\az_i}(f)(x):=\sup_{|x-y|<\az_i t,\,t\in(0,\fz)}\lf|\psi_i(t\sqrt{L})(f)(y)\r|.$$
Then there exists a positive constant $C$, depending on $n$, $\fai$, $\psi_1$, $\psi_2$,
$\az_1$ and $\az_2$, such that, for all $f\in L^2(\rn)$,
\begin{eqnarray}\label{2.2}
\lf\|\psi^{\ast}_{1,\,L,\,\az_1}(f)\r\|_{L^\fai(\rn)}\le C\lf\|\psi^{\ast}_{2,\,L,\,\az_2}(f)\r\|_{L^\fai(\rn)}.
\end{eqnarray}
Specially, for any even function $\phi\in\cs(\rr)$ with $\phi(0)=1$ and $\az\in(0,\fz)$,
there exists a positive constant $C$, depending on $n$, $\fai$, $\phi$ and $\az$, such that, for all $f\in L^2(\rn)$,
\begin{eqnarray*}
C^{-1}\|f^\ast_L\|_{L^\fai(\rn)}\le\lf\|\phi^\ast_{L,\,\az}(f)\r\|_{L^\fai(\rn)}\le C\|f^\ast_L\|_{L^\fai(\rn)}.
\end{eqnarray*}
\end{proposition}

\begin{proof}
We first show that, for any $\psi\in\cs(\rn)$  and $p\in(q(\fai),\fz)$,
there exists a positive constant $C$, depending on $n$, $\fai$, $\psi$ and $p$, such that, for all
$0<\az_2\le \az_1$ and $f\in L^2(\rn)$,
\begin{eqnarray}\label{2.3}
\lf\|\psi^{\ast}_{L,\,\az_1}(f)\r\|_{L^\fai(\rn)}\le C\lf[\frac{\az_1}{\az_2}\r]^{np}
\lf\|\psi^{\ast}_{L,\,\az_2}(f)\r\|_{L^\fai(\rn)}.
\end{eqnarray}

For any $\lz\in(0,\fz)$, let
$$E_{\lz}:=\{x\in\rn:\ \psi^{\ast}_{L,\,\az_2}(f)(x)>\lz\}\ \text{and}\
E^\ast_\lz:=\{x\in\rn:\ \cm(\chi_{E_\lz})(x)>\wz{C}/(\az_1/\az_2)^n\},$$
where $\cm$ denotes the Hardy-Littlewood maximal operator on $\rn$
and $\wz{C}\in(0,1)$ is a positive constant. By $\fai\in\aa_\fz(\rn)$ and the definition of $q(\fai)$,
we know that, for any $p\in(q(\fai),\fz)$, $\fai\in\aa_p(\rn)$,
which, combined with Lemma \ref{l2.2}(iii), implies that, for all $t\in(0,\fz)$,
\begin{eqnarray}\label{2.4}
\int_{E^\ast_\lz}\fai(x,t)\,dx\ls\frac{(\az_1/\az_2)^{np}}{\wz{C}^p}\int_{E_\lz}\fai(x,t)\,dx,
\end{eqnarray}
where the implicit constant depends on $n$, $p$ and $\fai$.

Furthermore, we claim that $\psi^\ast_{L,\,\az_1}(f)(x)\le\lz$ for all $x\not\in
E^\ast_\lz$. Indeed, let $x\in\rn\setminus E^\ast_\lz$ and fix any given $(y,t)\in\rr^{n+1}_+:=\rn\times(0,\fz)$
satisfying $|y-x|<\az_1t$. Then $B(y,\az_2t)\not\subset E_\lz$. If this is
not true, then
$$\cm(\chi_{E_\lz})(x)\ge\frac{|B(y,\az_2t)|}{|B(y,\az_1t)|}=\lf(\frac{\az_2}{\az_1}\r)^n
>\frac{\wz{C}}{(\az_1/\az_2)^n},
$$
which gives a contradiction with $x\not\in E^\ast_\lz$, and hence
the claim holds true. By this claim, we conclude that there exists $z\in
B(y,\az_2t)$ such that $\psi^\ast_{L,\,\az_2}(f)(z)\le\lz$, which implies that
$|\psi_{L}(f)(y)|\le\psi^\ast_{L,\,\az_2}(f)(z)\le\lz$, where $\psi_L(f)(y):=\psi(t\sqrt{L})(f)(y)$.
From this and the choice of $(y,t)$, we deduce that, for all $x\not\in E^\ast_\lz$,
$\psi^\ast_{L,\,\az_1}(f)(x)\le\lz$, which, together with Lemma \ref{l2.2}(ii),
Fubini's theorem and \eqref{2.4},  further implies that
\begin{eqnarray*}
\int_\rn\fai\lf(x,\psi^\ast_{L,\,\az_1}(f)(x)\r)\,dx&&\sim\int_\rn\int_0^{\psi^\ast_{L,\,\az_1}(f)(x)}
\frac{\fai(x,t)}{t}\,dt\,dx\\ &&\sim\int_0^\fz\int_{\{x\in\rn:\
\psi^\ast_{L,\,\az_1}(f)(x)>t\}}\frac{\fai(x,t)}{t}\,dx\,dt\\
&&\ls\int_0^\fz\int_{E^\ast_t}\frac{\fai(x,t)}{t}\,dx\,dt\ls\lf(\frac{\az_1}{\az_2}\r)^{np}\int_0^\fz
\int_{E_t}\frac{\fai(x,t)}{t}\,dx\,dt\\
&&\sim\lf(\frac{\az_1}{\az_2}\r)^{np}\int_\rn\fai\lf(x,\psi^\ast_{L,\,\az_2}(f)(x)\r)\,dx,
\end{eqnarray*}
where the implicit positive constant depend on $n$, $p$, $\psi$ and $\fai$.
By this, we know that, for all $\lz\in(0,\fz)$,
$$\int_\rn\fai\lf(x,\frac{\psi^\ast_{L,\,\az_1}(f)(x)}{\lz}\r)\,dx\ls
\lf(\frac{\az_1}{\az_2}\r)^{np}\int_\rn\fai\lf(x,\frac{\psi^\ast_{L,\,\az_2}(f)(x)}{\lz}\r)\,dx,
$$
which further implies that \eqref{2.3} holds true.

Let $\psi:=\psi_1-\psi_2$. Via \eqref{2.3}, to prove \eqref{2.2}, it suffices to show that
\begin{eqnarray}\label{2.5}
\lf\|\psi^{\ast}_{L,\,1}(f)\r\|_{L^\fai(\rn)}\ls\lf\|\psi^{\ast}_{2,\,L,\,1}(f)\r\|_{L^\fai(\rn)},
\end{eqnarray}
where the implicit constant depends on $n$, $\psi_1$, $\psi_2$ and $\fai$.
Now we prove \eqref{2.5}. Let $\Psi(x):=x^{2k}\Phi(x)$ for all $x\in\rn$, where $\Phi$ is as in Lemma \ref{l2.1} and $k\in\nn$
with $k>nq(\fai)/2i(\fai)$. By the spectral calculus, we know that there exists a constant $C_{(\Psi,\,\psi_2)}$,
depending on $\Psi$ and $\psi_2$, such that
$$f=C_{(\Psi,\,\psi_2)}\int_0^\fz\Psi(s\sqrt{L})\psi_2(s\sqrt{L})(f)\,\frac{ds}{s},
$$
which further implies that, for any $t\in(0,\fz)$,
\begin{eqnarray*}
\psi(t\sqrt{L})(f)=C_{(\Psi,\,\psi_2)}\int_0^\fz\lf[\psi(t\sqrt{L})\Psi(s\sqrt{L})\r]\psi_2(s\sqrt{L})(f)\,\frac{ds}{s}.
\end{eqnarray*}
Denote by $K_{\psi(t\sqrt{L})\Psi(s\sqrt{L})}$ the kernel of $\psi(t\sqrt{L})\Psi(s\sqrt{L})$.
Then, for all $\lz\in(0,\fz)$, we have
\begin{eqnarray}\label{2.6}
&&\sup_{|w|<t,\,t\in(0,\fz)}\lf|\psi(t\sqrt{L})(f)(x-w)\r|\\ \nonumber
&&\hs\sim\sup_{|w|<t,\,t\in(0,\fz)}\lf|\int_0^\fz
K_{\psi(t\sqrt{L})\Psi(s\sqrt{L})}(x-w,z)\psi_2(s\sqrt{L})(f)(z)\frac{dz\,ds}{s}\r|\\ \nonumber
&&\hs\ls\sup_{|w|<t,\,t\in(0,\fz)}\int_{\rr^{n+1}_+}\lf|K_{\psi(t\sqrt{L})\Psi(s\sqrt{L})}(x-w,z)\r|\lf[1+\frac{|x-z|}{s}\r]^{\lz}\\ \nonumber
&&\hs\hs\times\lf|\psi_2(s\sqrt{L})(f)(z)\r|\lf[1+\frac{|x-z|}{s}\r]^{-\lz}\frac{dz\,ds}{s}\\ \nonumber
&&\hs\ls\sup_{z,\,s}\lf|\psi_2(s\sqrt{L})(f)(z)\r|\lf[1+\frac{|x-z|}{s}\r]^{-\lz}\\ \nonumber
&&\hs\hs\times\sup_{|w|<t,\,t\in(0,\fz)}\int_{\rr^{n+1}_+}\lf|K_{\psi(t\sqrt{L})\Psi(s\sqrt{L})}(x-w,z)\r|
\lf[1+\frac{|x-z|}{s}\r]^{\lz}\frac{dz\,ds}{s}.
\end{eqnarray}
Moreover, from \cite[(3.4)]{sy15}, it follows that, for any $\lz\in(0,2k)$,
$$\sup_{|w|<t,\,t\in(0,\fz)}\int_{\rr^{n+1}_+}\lf|K_{\psi(t\sqrt{L})
\Psi(s\sqrt{L})}(x-w,z)\r|\lf[1+\frac{|x-z|}{s}\r]^{\lz}\frac{dz\,ds}{s}\ls1,
$$
where the implicit positive constant depends on $n$, $\Psi$, $\psi$ and $\lz$,
which, combined with \eqref{2.6}, implies that
\begin{eqnarray}\label{2.7}
\sup_{|w|<t,\,t\in(0,\fz)}\lf|\psi(t\sqrt{L})f(x-w)\r|\ls\sup_{z,\,s}\lf|\psi_2(s\sqrt{L})
(f)(z)\r|\lf[1+\frac{|x-z|}{s}\r]^{-\lz},
\end{eqnarray}
where the implicit constant depends on $n$, $\Psi$, $\psi$ and $\lz$.
Furthermore, let $\chi$ be the characteristic function of $[0,1]$. Then, for all $\lz,\,s\in(0,\fz)$,
we have
\begin{eqnarray*}
(1+s)^{-\lz}\le\sum_{k=1}^\fz2^{-k}\chi\lf(\frac{s}{2^{k/\lz}}\r),
\end{eqnarray*}
which further implies that
\begin{eqnarray}\label{2.8}
&&\sup_{z,\,s}\lf|\psi_2(s\sqrt{L})(f)(z)\r|\lf[1+\frac{|x-z|}{s}\r]^{-\lz}\\ \nonumber
&&\hs\le\sum_{k=1}^\fz2^{-k}\sup_{z,\,s}\lf|\psi_2(s\sqrt{L})(f)(z)\r|\chi\lf(\frac{|x-z|}{s2^{k/\lz}}\r)\\ \nonumber
&&\hs=\sum_{k=1}^\fz2^{-k}\sup_{|x-z|<2^{k/\lz}s,\,s\in(0,\fz)}\lf|\psi_2(s\sqrt{L})(f)(z)\r|
=\sum_{k=1}^\fz2^{-k}\psi^\ast_{2,\,L,\,2^{k/\lz}}(f)(x).
\end{eqnarray}
Let $\lz\in(nq(\fai)/i(\fai),2k)$. Then, by $\lz>nq(\fai)/i(\fai)$ and the definitions of
$q(\fai)$ and $i(\fai)$, we conclude that there exist $p_0\in(0,i(\fai))$
and $\wz{q}\in(q(\fai),\fz)$ such that $\lz>n\wz{q}/p_0$, $\fai$ is of
uniformly lower type $p_0$ and $\fai\in\aa_{\wz{q}}(\rn)$,
which, together with \eqref{2.8}, Lemma \ref{l2.2}(i) and \eqref{2.3},
further implies that, for all $\mu\in(0,\fz)$,
\begin{eqnarray*}
&&\int_\rn\fai\lf(x,\sup_{z,\,s}\lf|\psi_2(s\sqrt{L})(f)(z)\r|\lf[1+\frac{|x-z|}{s}\r]^{-\lz}/\mu\r)\,dx\\
&&\hs\le\int_\rn\fai\lf(x,\sum_{k=1}^\fz2^{-k}\psi^\ast_{2,\,L,\,2^{k/\lz}}(f)(x)/\mu\r)\,dx\\
&&\hs\ls\sum_{k=1}^\fz2^{-kp_0}\int_\rn\fai\lf(x,\psi^\ast_{2,\,L,\,2^{k/\lz}}(f)(x)/\mu\r)\,dx\\
&&\hs\ls\sum_{k=1}^\fz2^{-k[p_0-n\wz{q}/\lz]}\int_\rn\fai\lf(x,\psi^\ast_{2,\,L}(f)(x)/\mu\r)\,dx
\sim\int_\rn\fai\lf(x,\psi^\ast_{2,\,L}(f)(x)/\mu\r)\,dx,
\end{eqnarray*}
where the implicit constants depend on $n$, $\psi$ and $\fai$.
From this, it follows that
$$\lf\|\sup_{z,\,s}\lf|\psi_2(s\sqrt{L})(f)(z)\r|\lf[1+\frac{|\cdot-z|}{s}\r]^{-\lz}\r\|_{L^\fai(\rn)}
\ls\lf\|\psi^\ast_{2,\,L}(f)\r\|_{L^\fai(\rn)},$$
which, combined with \eqref{2.7}, further implies that
\begin{eqnarray}\label{2.9}
\lf\|\psi^\ast_{L,\,1}(f)\r\|_{L^\fai(\rn)}&&=\lf\|\sup_{|w|<t,\,t\in(0,\fz)}
\lf|\psi(t\sqrt{L})(f)(\cdot-w)\r|\r\|_{L^\fai(\rn)}\\ \nonumber
&&\ls\lf\|\sup_{z,\,s}\lf|\psi_2(s\sqrt{L})(f)(z)\r|\lf[1+\frac{|\cdot-z|}{s}\r]^{-\lz}\r\|_{L^\fai(\rn)}\\ \nonumber
&&\ls\lf\|\psi^\ast_{2,\,L}(f)\r\|_{L^\fai(\rn)},
\end{eqnarray}
where the implicit constants depend on $n$, $\psi$, $\Psi$, $\lz$ and $\fai$.
This finishes the proof of \eqref{2.5} and hence Proposition \ref{p2.1}.
\end{proof}

Furthermore, to show Theorem \ref{t1.1}, we also need the following atomic characterization of the
Musielak-Orlicz-Hardy space $H_{\fai,\,L}(\rn)$ obtained in \cite[Theorem 5.4]{bckyy13b}.

\begin{proposition}\label{p2.2}
Let $L$ satisfy the assumptions $(A1)$
and $(A2)$ and $\fai$ be as in Definition \ref{d1.2}.
Assume that $M\in\nn\cap(\frac{nq(\fai)}{2i(\fai)},\fz)$ and
$q\in([r(\fai)]'I(\fai),\fz)$, where $q(\fai)$, $i(\fai)$, $r(\fai)$ and
$I(\fai)$ are, respectively, as in \eqref{1.10}, \eqref{1.9},
\eqref{1.11} and \eqref{1.8}.
Then the spaces $H_{\fai,\,L}(\rn)$ and
$H^{M,\,q}_{\fai,\,L,\,\mathrm{at}}(\rn)$ coincide with equivalent quasi-norms.
\end{proposition}

Now we give out the proof of Theorem \ref{t1.1} via Propositions \ref{p2.1} and \ref{p2.2}.

\begin{proof}[Proof of Theorem \ref{t1.1}]
We first show (i) of Theorem \ref{t1.1}. To this end, we begin with  proving that
\begin{eqnarray}\label{2.10}
\lf[H^{\phi,\,\az}_{\fai,\,L,\,\mathrm{max}}(\rn)\cap L^2(\rn)\r]
\subset \lf[H^{M,\,\fz}_{\fai,\,L,\,\mathrm{at}}(\rn)\cap L^2(\rn)\r].
\end{eqnarray}

To prove \eqref{2.10}, via Proposition \ref{p2.1}, it suffices to show that,
for any $f\in H_{\fai,\,L,\,\mathrm{max}}(\rn)\cap L^2(\rn)$,
$f\in H^{M,\,\fz}_{\fai,\,L,\,\mathrm{at}}(\rn)$ and
\begin{eqnarray}\label{2.11}
\|f\|_{H^{M,\,\fz}_{\fai,\,L,\,\mathrm{at}}(\rn)}\ls\|f\|_{H_{\fai,\,L,\,\mathrm{max}}(\rn)},
\end{eqnarray}
where the implicit positive constant depends on $n$, $M$ and $\fai$.
Let $\Psi(x):=x^{2M}\Phi(x)$ for all $x\in\rn$, where $\Phi$ is as in Lemma \ref{l2.1}.
Then, by the spectral calculus, we know that there exists a constant $C_{(\Psi)}$ such that
\begin{eqnarray*}
f=C_{(\Psi)}\int_0^\fz\Psi(t\sqrt{L})t^2Le^{-t^2L}(f)\frac{dt}{t}
\end{eqnarray*}
holds true in $L^2(\rn)$. For $x\in\rr$, let
$$\eta(x):=\begin{cases}C_{(\Psi)}\dint_1^\fz t^2x\Psi(tx)e^{-t^2x^2}\frac{dt}{t}, & x\neq0,\\
1, & x=0.
\end{cases}
$$
Then $\eta\in\cs(\rr)$ is an even function and, for any $a,\,b\in\rr$,
$$\eta(ax)-\eta(bx)=C_{(\Psi)}\int_a^b t^2x^2\Psi(tx)e^{-t^2x^2}\,\frac{dt}{t},
$$
which further implies that
\begin{eqnarray*}
C_{(\Psi)}\int_a^b\Psi(t\sqrt{L})t^2Le^{-t^2L}(f)\,\frac{dt}{t}=\eta(a\sqrt{L})(f)-\eta(b\sqrt{L})(f).
\end{eqnarray*}

For any $f\in L^2(\rn)$ and $x\in\rn$, let
$$\cm_L(f)(x):=\sup_{|x-y|<5\sqrt{n}t,\,t\in(0,\fz)}\lf[\lf|t^2Le^{-t^2L}(f)(y)\r|+\lf|\eta(t\sqrt{L})(f)(y)\r|\r].
$$
Then, from Proposition \ref{p2.1}, it follows that
\begin{eqnarray}\label{2.12}
\|\cm_L(f)\|_{L^\fai(\rn)}\ls\|f\|_{H_{\fai,\,L,\,\mathrm{max}}(\rn)},
\end{eqnarray}
where the implicit positive constant depends on $n$, $M$, $\fai$ and $\Phi$ as in Lemma \ref{l2.1}.
Let
$$\widehat{O}:=\{(x,t)\in\rr_+^{n+1}:\ B(x,4\sqrt{n}t)\subset O\}$$
and, for any $i\in\zz$,
$$O_i:=\lf\{x\in\rn:\ \cm_L(f)(x)>2^i\r\}.$$
For each $i\in\zz$, denote by $\{Q_{i,j}\}_{j\in\nn}$ the Whitney decomposition of $O_i$.
For each $i\in\zz$ and $j\in\nn$, let
$$\widetilde{Q}_{i,j}:=\{(y,t)\in\rr^{n+1}_+:\ y+3t\overline{e}\in Q_{i,j}\},
$$
here and hereafter, $\overline{e}:=(1,\,\ldots,\,1)\in\rn$.
It is easy to prove that, for all $i\in\zz$, $\widehat{O}_i\subset\bigcup_j\widetilde{Q}_{i,j}$.
Indeed, for any $(y^\circ,t^\circ)\in\widehat{O}_i$, $B(y^\circ,4\sqrt{n}t^\circ)\subset O_i$.
Let $\widetilde{y}^\circ:=y^\circ+3\overline{e}t^\circ$. Then $\widetilde{y}^\circ\in B(y^\circ,4\sqrt{n}t^\circ)\subset O_i$,
which implies that there exists $Q_{i,j_0}\subset O_i$ such that $\widetilde{y}^\circ\in Q_{i,j_0}$.
By this, we conclude that $(y^\circ,t^\circ)\in \widetilde{Q}_{i,j_0}$ and hence $\widehat{O}_i\subset\bigcup_j\widehat{Q}_{i,j}$.

Furthermore, notice that, for any $j_1\neq j_2$, $\widetilde{Q}_{i,j_1}\cap\widetilde{Q}_{i,j_2}=\emptyset$
and
$$\rr^{n+1}_+=\bigcup_i\widehat{O}_i=\bigcup_i(\widehat{O}_i\backslash\widehat{O}_{i+1})=\bigcup_i\bigcup_j T_{i,j},$$
where $T_{i,j}:=\widetilde{Q}_{i,j}\cap(\widehat{O}_i\backslash\widehat{O}_{i+1})$.
Thus,
\begin{eqnarray}\label{2.13}
f=\sum_{i,\,j}C_{(\Psi)}\int_0^\fz\Psi(t\sqrt{L})\lf(\chi_{T_{i,j}}t^2Le^{-t^2L}(f)\r)\,\frac{dt}{t}=:\sum_{i,\,j}\lz_{i,j}\az_{i,j},
\end{eqnarray}
where $\lz_{i,j}:=2^i\|\chi_{Q_{i,j}}\|_{L^\fai(\rn)}$ and $\az_{i,j}:=L^M b_{i,j}$ with
$$b_{i,j}:=\frac{C_{(\Psi)}}{\lz_{i,j}}\int_0^\fz t^{2M}\Phi(t\sqrt{L})\lf(\chi_{T_{i,j}}t^2Le^{-t^2L}(f)\r)\,\frac{dt}{t}.
$$
Now we prove that the summation \eqref{2.13} converges in $L^2(\rn)$. Indeed, it is well known that, for any $f\in L^2(\rn)$,
$$\lf\{\int_{\rr^{n+1}_+}\lf|t^2Le^{-t^2L}(f)(y)\r|^2\frac{dy\,dt}{t}\r\}^{1/2}\ls\|f\|_{L^2(\rn)}
$$
(see, for example, \cite[(3.14)]{hlmmy}), which, together with \eqref{2.13}, implies that
\begin{eqnarray*}
&&\lf\|\sum_{|i|>N_1,\,|j|>N_2}\lz_{i,j}\az_{i,j}\r\|_{L^2(\rn)}\\
&&\hs\sim
\lf\|\sum_{|i|>N_1,\,|j|>N_2}\int_{\rr^{n+1}_+}K_{(t^2L)^M\Phi(t\sqrt{L})}(\cdot,y)\chi_{T_{i,j}}(y,t)t^2Le^{-t^2L}
(f)(y)\frac{dy\,dt}{t}\r\|_{L^2(\rn)}\\
&&\hs\ls\sup_{\|g\|_{L^2(\rn)}\le1}\sum_{|i|>N_1,\,|j|>N_2}\int_{T_{i,j}}\lf|(t^2L)^M\Phi(t\sqrt{L})(g)(y)t^2Le^{-t^2L}
(f)(y)\r|\frac{dy\,dt}{t}\\
&&\hs\ls\lf\{\sum_{|i|>N_1,\,|j|>N_2}\int_{T_{i,j}}\lf|t^2Le^{-t^2L}(f)(y)\r|^2\frac{dy\,dt}{t}\r\}^{1/2}\rightarrow0,
\end{eqnarray*}
as $N_1\to\fz$ and $N_2\to\fz$. Thus, the summation \eqref{2.13} converges in $L^2(\rn)$.

Now we claim that there exists a positive constant $\wz{C}$,
depending on $n$, $M$, $\Phi$ and $\fai$, such that, for all $i$ and $j$,
$\wz{C}^{-1}\az_{i,j}$ is a $(\fai,\,\fz,\,M)_L$-atom associated with the ball $30B_{i,j}$,
where $B_{i,j}$ denotes the ball with the center
being the same as $Q_{i,j}$ and the radius $r_{B_{i,j}}:=\sqrt{n}\ell(Q_{i,j})/2$.
Here and hereafter, $\ell(Q_{i,j})$ denotes the side length of $Q_{i,j}$.
Once this claim is proved, by Lemma \ref{l2.2}(iv), we then know that, for all $\lz\in(0,\fz)$,
\begin{eqnarray}\label{2.14}
\sum_{i,j}\fai\lf(30B_{i,j},\frac{\lz_{i,j}}{\lz\|\chi_{30B_{i,j}}\|_{L^\fai(\rn)}}\r)
&&\ls\sum_{i,j}\fai\lf(Q_{i,j},\frac{2^i}{\lz}\r)\sim\sum_{i}\fai\lf(O_{i},\frac{2^i}{\lz}\r).
\end{eqnarray}
Moreover, similar to the proof of \cite[Lemma 5.4]{k} (see also \cite[Lemma 3.4]{yys4}), we conclude that
$$
\sum_{i}\fai\lf(O_{i},\frac{2^i}{\lz}\r)\ls\int_{\rn}\fai\lf(x,\frac{\cm_L(f)(x)}{\lz}\r)\,dx,$$
which, combined with \eqref{2.12} and \eqref{2.14}, further implies that \eqref{2.11} holds true.

Now we prove the above claim. We first show that, for any $k\in\{0,\,1,\,\ldots,\,M\}$,
\begin{eqnarray}\label{2.15}
\supp(L^kb_{i,j})\subset30Q_{i,j}\subset30B_{i,j}.
\end{eqnarray}
From the definition of $T_{i,j}$, it follows that, if $(y,t)\in T_{i,j}$, then $B(y,4\sqrt{n}t)\subset O_i$.
Let $\widetilde{y}:=y+3t\overline{e}$. Then $\widetilde{y}\in Q_{i,j}$ and $B(\widetilde{y},\sqrt{n}t)\subset O_i$.
Moreover, by the fact that $Q_{i,j}$ is the Whitney cube of $O_i$, we know that $5Q_{i,j}\cap O_i^\complement\neq\emptyset$
and hence $t\le3\ell(Q_{i,j})$, which, together with $y+3t\overline{e}\in Q_{i,j}$, further implies that $y\in20Q_{i,j}$.
Furthermore, from Lemma \ref{l2.1}, we deduce that $\supp(K_{(t^2L)^k\Phi(t\sqrt{L})})\subset\{(x,y)\in\rn\times\rn:\ |x-y|\le t\}$,
which, combined with $y\in20Q_{i,j}$, implies that \eqref{2.15} holds true.
To finish the proof of the above claim, it remains to prove that, for any $k\in\{0,\,1,\,\ldots,\,M\}$,
\begin{eqnarray}\label{2.16}
\lf\|([30r_{B_{i,j}}]^2L)^kb_{i,j}\r\|_{L^\fz(\rn)}\le \wz{C}(30r_{B_{i,j}})^{2M}\|\chi_{30B_{i,j}}\|^{-1}_{L^\fai(\rn)}.
\end{eqnarray}
By \cite[(3.12)]{sy15}, we find that, for any $k\in\{0,\,1,\,\ldots,\,M-1\}$,
$$\lf|\int_0^\fz\int_{\rn}K_{t^{2M}L^k\Phi(t\sqrt{L})}(x,y)\chi_{T_{i,j}}(y,t)t^2Le^{-t^2L}(f)(y)\frac{dy\,dt}{t}\r|
\ls2^{i}[\ell(Q_{i,j})]^{2(M-k)},
$$
where the implicit positive constant depends on $n$, $M$ and $\Phi$,
which further implies that
\begin{eqnarray}\label{2.17}
\hs\hs\lf\|([30r_{B_{i,j}}]^2L)^kb_{i,j}\r\|_{L^\fz(\rn)}&&=\lz_{i,j}^{-1}(30r_{B_{i,j}})^{2k}C_{(\Psi)}
\lf\|\int_{\rr^{n+1}_+}K_{t^{2M}L^k\Phi(t\sqrt{L})}(x,y)\r.\\ \nonumber
&&\hs\times\chi_{T_{i,j}}(y,t)t^2Le^{-t^2L}(f)(y)\frac{dy\,dt}{t}\Bigg\|_{L^\fz(\rn)}\\ \nonumber
&&\le \wz{C}\lz_{i,j}^{-1}2^i(30r_{B_{i,j}})^{2k}(r_{B_{i,j}})^{2(M-k)}\\ \nonumber
&&\le \wz{C}(30r_{B_{i,j}})^{2M}\|\chi_{30{B_{i,j}}}\|^{-1}_{L^\fai(\rn)}.
\end{eqnarray}
Furthermore, it follows, from \cite[(3.13)]{sy15}, that
$$\lf|\int_0^\fz\int_{\rn}K_{\Psi(t\sqrt{L})}(x,y)\chi_{T_{i,j}}(y,t)t^2Le^{-t^2L}(f)(y)\frac{dy\,dt}{t}\r|\ls2^{i},
$$
where the implicit positive constant depends on $n$ and $\Psi$, which implies that
\begin{eqnarray*}
\hs\hs\lf\|([30r_{B_{i,j}}]^2L)^Mb_{i,j}\r\|_{L^\fz(\rn)}&&=\lz_{i,j}^{-1}(30r_{B_{i,j}})^{2M}C_{(\Psi)}
\lf\|\int_{\rr^{n+1}_+}K_{\Psi(t\sqrt{L})}(x,y)\r.\\ \nonumber
&&\hs\times\chi_{T_{i,j}}(y,t)t^2Le^{-t^2L}(f)(y)\frac{dy\,dt}{t}\Bigg\|_{L^\fz(\rn)}\\ \nonumber
&&\le \wz{C}\lz_{i,j}^{-1}(30r_{B_{i,j}})^{2M}2^i\le \wz{C}(30r_{B_{i,j}})^{2M}\|\chi_{30{B_{i,j}}}\|^{-1}_{L^\fai(\rn)}.
\end{eqnarray*}
By this and \eqref{2.17}, we conclude that \eqref{2.16} holds true, which completes the proof of
the above claim and hence \eqref{2.10}.

Now we prove that, for any $M\in\nn\cap(nq(\fai)/2i(\fai),\fz)$ and $q\in([r(\fai)]'I(\fai),\fz]$,
\begin{eqnarray}\label{2.18}
\lf[H^{M,\,q}_{\fai,\,L,\,\mathrm{at}}(\rn)\cap L^2(\rn)\r]
\subset \lf[H^{\ca}_{\fai,\,L,\,\mathrm{max}}(\rn)\cap L^2(\rn)\r].
\end{eqnarray}

For any $\phi\in\ca$ and $x\in\rr$, let $\wz{\psi}(x):=[\phi(0)]^{-1}\phi(x)-e^{-x^2}$.
Repeating the proof of \cite[(3.4)]{sy15}, we know that, for any $\lz\in(0,2M)$,
there exists a positive constant $C$, depending on $n$, $\Psi$ and $\lz$,
such that, for all $\phi\in\ca$,
$$\sup_{|w|<t,\,t\in(0,\fz)}
\int_{\rr^{n+1}_+}\lf|K_{\wz{\psi}(t\sqrt{L})\Psi(s\sqrt{L})}(x-w,z)\r|\lf[1+\frac{|x-z|}{s}\r]^{\lz}\frac{dz\,ds}{s}\le C,
$$
where $\Psi$ is as in \eqref{2.6}. Via this estimate and repeating the proof of \eqref{2.9},
we find that
\begin{eqnarray*}
\lf\|\sup_{\phi\in\ca}\wz{\psi}^\ast_{L,\,1}(f)\r\|_{L^\fai(\rn)}\ls\lf\|f^\ast_L\r\|_{L^\fai(\rn)},
\end{eqnarray*}
where the implicit constant depends on $n$, $\Psi$, $\lz$ and $\fai$,
which, combined with the fact that $\cg^\ast_{L}(f)\ls\sup_{\phi\in\ca}\wz{\psi}^\ast_{L,\,1}(f)+f^\ast_L$
and Lemma \ref{l2.2}(i), further implies that
\begin{eqnarray}\label{2.19}
\lf\|\cg^\ast_{L}(f)\r\|_{L^\fai(\rn)}\ls\lf\|f^\ast_L\r\|_{L^\fai(\rn)}.
\end{eqnarray}
Via \eqref{2.19}, to finish the proof of \eqref{2.18}, it suffices to show that, for any $\lz\in\cc$
and $(\fai,\,q,\,M)_L$-atom $\az$ associated with the ball $B:=B(x_B,r_B)$ with $x_B\in\rn$ and $r_B\in(0,\fz)$,
\begin{eqnarray}\label{2.20}
\int_{\rn}\fai\lf(x,|\lz|\az^\ast_L(x)\r)\,dx\ls\fai\lf(B,|\lz|\|\chi_{B}\|^{-1}_{L^\fai(\rn)}\r),
\end{eqnarray}
where the implicit positive constant depends on $n$ and $\fai$.
Indeed, let $f\in H^{M,\,q}_{\fai,\,L,\,\mathrm{at}}(\rn)\cap L^2(\rn)$. Then there exist
$\{\lz_j\}_j\subset\cc$ and a sequence $\{\az_j\}_j$
of $(\fai,\,q,\,M)_L$-atoms, associated with the balls $\{B_j\}_j$, such that
\begin{equation*}
f=\sum_{j}\lz_j\az_j\ \text{in}\ L^2(\rn) \ \text{and} \ \
\|f\|_{H^{M,\,q}_{\fai,\,L,\,\mathrm{at}}(\rn)}\sim\blz(\{\lz_j\az_j\}_j),
\end{equation*}
which, together with \eqref{2.20}, further implies that, for all $\lz\in(0,\fz)$,
\begin{eqnarray*}
\int_\rn\fai\lf(x,\frac{f^\ast_L(x)}{\lz}\r)\,dx&&\ls\sum_j\int_\rn
\fai\lf(x,\frac{|\lz_j|(\az_j)_L^\ast(x)}{\lz}\r)\,dx\\
&&\ls\sum_j\fai\lf(B_j,\frac{|\lz_j|}{\lz\|\chi_{B_j}\|_{L^\fai(\rn)}}\r).
\end{eqnarray*}
From this and \eqref{2.19}, it follows that $f\in H^{\ca}_{\fai,\,L,\,\mathrm{max}}(\rn)\cap L^2(\rn)$ and
$$\|f\|_{H^{\ca}_{\fai,\,L,\,\mathrm{max}}(\rn)}\ls\|f\|_{H^{M,\,q}_{\fai,\,L,\,\mathrm{at}}(\rn)}.$$

Now we prove \eqref{2.20}. By \eqref{1.5}, we conclude that, for all $x\in\rn$,
\begin{eqnarray}\label{2.21}
\az^\ast_L(x)\ls\cm(\az)(x),
\end{eqnarray}
where $\cm$ denotes the Hardy-Littlewood maximal operator on $\rn$. Moreover, from $q\in([r(\fai)]'I(\fai),\fz]$,
it follows that there exists $p_1\in [I(\fai),\,1]$ such that $\fai$ is of uniformly upper type $p_1$ and
$\fai\in \rh_{(q/p_1)'}(\rn)$, which, combined with \eqref{2.21}, H\"older's inequality,
the boundedness of $\cm$ on $L^{q}(\rn)$ and Lemma \ref{l2.2}(iv),
further implies that
\begin{eqnarray}\label{2.22}
&&\int_{4B}\fai\lf(x,|\lz|\az^\ast_L(x)\r)\,dx\\ \nonumber
&&\hs\ls\int_{4B}\fai\lf(x,|\lz|\cm(\az)(x)\r)\,dx\\ \nonumber
&&\hs\ls\int_{4B}\fai\lf(x,|\lz|\|\chi_{B}\|^{-1}_{L^\fai(\rn)}\r)
\lf[1+\cm(\az)(x)\|\chi_{B}\|_{L^\fai(\rn)}\r]^{p_1}\,dx\\ \nonumber
&&\hs\ls\fai\lf(4B,|\lz|\|\chi_{B}\|_{L^\fai(\rn)}^{-1}\r)+\|\chi_{B}\|_{L^\fai(\rn)}^{p_1}
\lf\|\cm(\az)\r\|_{L^q(4B)}^{p_1}\\ \nonumber
&&\hs\hs\times\lf\{\int_{4B} \lf[\fai\lf(x,\, |\lz|\|\chi_{B}\|_{L^\fai(\rn)}^{-1}\r)\r]^{(\frac{q}{p_1})'}\,dx\r\}^{\frac{1}{(\frac{q}{p_1})'}}
\ls\fai\lf(B,|\lz|\|\chi_{B}\|_{L^\fai(\rn)}^{-1}\r).
\end{eqnarray}
For $x\in\rn\backslash(4B)$, let
$$\az^\ast_{1}(x):=\sup_{|x-y|<t,\,t\in(0,r_{B}]}\lf|e^{-t^2L}(\az)(y)\r|,
$$
$$\az^\ast_{2}(x):=\sup_{|x-y|<t,\,t\in(r_{B},|x-x_{B}|/4]}\lf|e^{-t^2L}(\az)(y)\r|
$$
and
$$\az^\ast_{3}(x):=\sup_{|x-y|<t,\,t\in[|x-x_{B}|/4,\fz)}\lf|e^{-t^2L}(\az)(y)\r|.
$$
For any $t\in(0,|x-x_B|/4]$, $z\in B$ and $y\in\rn$ satisfying $|x-y|<t$,
we find that
\begin{eqnarray}\label{2.23}
|y-z|\ge|x-z|-|x-y|\ge|x-x_{B}|-r_{B}-t\ge\frac{|x-x_{B}|}{2},
\end{eqnarray}
which, together with \eqref{1.5}, implies that, for any $s\in(0,\fz)$,
\begin{eqnarray}\label{2.24}
\az^\ast_{1}(x)&&\ls\sup_{|x-y|<t,\,t\in(0,r_{B}]}\int_{B}\frac{t^s}{(t+|z-y|)^{n+s}}|\az(z)|\,dz\\ \nonumber
&&\ls\frac{r_B^s}{|x-x_B|^{n+s}}\|\az\|_{L^1(\rn)}\ls\frac{r_B^{n+s}}{|x-x_B|^{n+s}}\|\chi_B\|^{-1}_{L^\fai(\rn)}.
\end{eqnarray}
Moreover, from $\az=L^Mb$, \eqref{2.23} and \eqref{1.5}, it follows that, for any $s\in(0,2M)$,
\begin{eqnarray}\label{2.25}
\hs\hs\az^\ast_{2}(x)&&\ls\sup_{|x-y|<t,\,t\in(r_{B},|x-x_{B}|/4]}t^{-2M}\lf|(t^2L)^Me^{-t^2L}(b)(y)\r|\\ \nonumber
&&\ls\sup_{|x-y|<t,\,t\in(r_{B},|x-x_{B}|/4]}t^{-2M}\int_{B}\frac{t^s}{(t+|z-y|)^{n+s}}|b(z)|\,dz\\ \nonumber
&&\ls\sup_{t\in(r_{B},|x-x_{B}|/4]}t^{s-2M}|x-x_B|^{-n-s}\|b\|_{L^1(\rn)}
\ls\frac{r_B^{n+s}}{|x-x_B|^{n+s}}\|\chi_B\|^{-1}_{L^\fai(\rn)}.
\end{eqnarray}
Furthermore, by $\az=L^Mb$ and \eqref{1.5}, we conclude that, for any $s\in(0,2M)$,
\begin{eqnarray*}
\az^\ast_{3}(x)&&\ls\sup_{|x-y|<t,\,t\in[|x-x_{B}|/4,\fz)}t^{-2M}\int_{B}\frac{t^s}{(t+|z-y|)^{n+s}}|b(z)|\,dz\\ \nonumber
&&\ls\sup_{|x-y|<t,\,t\in(r_{B},|x-x_{B}|/4]}t^{-2M-n}\|b\|_{L^1(\rn)}
\ls\frac{r_B^{n+s}}{|x-x_B|^{n+s}}\|\chi_B\|^{-1}_{L^\fai(\rn)},
\end{eqnarray*}
which, combined with \eqref{2.24} and \eqref{2.25}, further implies that, for any $s\in(0,2M)$,
\begin{eqnarray}\label{2.26}
\az^\ast_L(x)\ls\frac{r_B^{n+s}}{|x-x_B|^{n+s}}\|\chi_B\|^{-1}_{L^\fai(\rn)}.
\end{eqnarray}
Let $s\in(nq(\fai)/i(\fai),2M)$. From $s>nq(\fai)/i(\fai)$, we deduce that there exist $p_0\in(0,i(\fai))$
and $\wz{q}\in(q(\fai),\fz)$ such that $s>n\wz{q}/p_0$, $\fai$ is
of uniformly lower type $p_0$ and $\fai\in\aa_{\wz{q}}(\rn)$,
which, together with \eqref{2.26} and Lemma \ref{l2.2}(iv), implies that
\begin{eqnarray*}
\int_{\rn\backslash(4B)}\fai\lf(x,|\lz|\az^\ast_L(x)\r)\,dx&&
=\sum_{j=2}^\fz\int_{S_j(B)}\fai\lf(x,|\lz|\az^\ast_L(x)\r)\,dx\\
&&\ls\sum_{j=2}^\fz2^{-j(n+s)p_0}\fai\lf(S_j(B),|\lz|\|\chi_B\|^{-1}_{L^\fai(\rn)}\r)\\
&&\ls\sum_{j=2}^\fz2^{-j[(n+s)p_0-n\wz{q}]}\fai\lf(B,|\lz|\|\chi_B\|^{-1}_{L^\fai(\rn)}\r)\\
&&\ls\fai\lf(B,|\lz|\|\chi_B\|^{-1}_{L^\fai(\rn)}\r).
\end{eqnarray*}
By this and \eqref{2.22}, we conclude that \eqref{2.20} holds true, which completes the proof of \eqref{2.18}.

By the definitions of the spaces $H^{\phi,\,\az}_{\fai,\,L,\,\mathrm{max}}(\rn)$ and $H^{\ca}_{\fai,\,L,\,\mathrm{max}}(\rn)$,
we know that
\begin{eqnarray}\label{2.27}
\lf[H^{\ca}_{\fai,\,L,\,\mathrm{max}}(\rn)\cap L^2(\rn)\r]\subset\lf[H^{\phi,\,\az}_{\fai,\,L,\,\mathrm{max}}(\rn)\cap L^2(\rn)\r].
\end{eqnarray}
Moreover, from the definition of the space $H^{M,\,q}_{\fai,\,L,\,\mathrm{at}}(\rn)$, with $q\in([r(\fai)]'I(\fai),\fz]$,
and Proposition \ref{p2.2}, it follows that, for any $q\in([r(\fai)]'I(\fai),\fz)$,
\begin{eqnarray*}
\lf[H^{M,\,\fz}_{\fai,\,L,\,\mathrm{at}}(\rn)\cap L^2(\rn)\r]\subset\lf[H^{M,\,q}_{\fai,\,L,\,\mathrm{at}}(\rn)\cap L^2(\rn)\r],
\end{eqnarray*}
which, combined with \eqref{2.10}, \eqref{2.18} and \eqref{2.27}, implies that, for any $q\in([r(\fai)]'I(\fai),\fz]$,
\begin{eqnarray*}
\lf[H^{M,\,q}_{\fai,\,L,\,\mathrm{at}}(\rn)\cap L^2(\rn)\r]=\lf[H^{\phi,\,\az}_{\fai,\,L,\,\mathrm{max}}(\rn)\cap L^2(\rn)\r]=\lf[H^{\ca}_{\fai,\,L,\,\mathrm{max}}(\rn)\cap L^2(\rn)\r]
\end{eqnarray*}
with equivalent quasi-norms, which, together with the fact that
$H^{M,\,q}_{\fai,\,L,\,\mathrm{at}}(\rn)\cap L^2(\rn)$,
$H^{\phi,\,\az}_{\fai,\,L,\,\mathrm{max}}(\rn)\cap L^2(\rn)$ and
$H^{\ca}_{\fai,\,L,\,\mathrm{max}}(\rn)\cap L^2(\rn)$ are, respectively, dense in the spaces
$H^{M,\,q}_{\fai,\,L,\,\mathrm{at}}(\rn)$,
$H^{\phi,\,\az}_{\fai,\,L,\,\mathrm{max}}(\rn)$ and $H^{\ca}_{\fai,\,L,\,\mathrm{max}}(\rn)$, and a
density argument, implies that the spaces $H^{M,\,q}_{\fai,\,L,\,\mathrm{at}}(\rn)$,
$H^{\phi,\,\az}_{\fai,\,L,\,\mathrm{max}}(\rn)$ and $H^{\ca}_{\fai,\,L,\,\mathrm{max}}(\rn)$
coincide with equivalent quasi-norms. This finishes the proof of Theorem \ref{t1.1}(i).

Furthermore, (ii) of Theorem \ref{t1.1} is deduced from (i) and Proposition \ref{p2.2}, which
completes the proof of Theorem \ref{t1.1}.
\end{proof}

\section{Proof of Theorem \ref{t1.2}\label{s3}}

\hskip\parindent In this section, we show Theorem \ref{t1.2}. We first introduce some notation.

Let $f\in L^2(\rn)$. For all $t\in(0,\fz)$ and $x\in\rn$, let
\begin{equation}\label{3.1}
u(x,t):=e^{-tL}(f)(x).
\end{equation}
For all $\uc\in(0,\fz)$, $N\in\nn$ and $x\in\rn$, define
\begin{eqnarray}\label{3.2}
u^{\ast}_{\uc,\,N}(x):=\sup_{|y-x|<\sqrt{t}<\uc^{-1},\,t\in(0,\fz)}|u(y,t)|\lf[\frac{\sqrt t}{\sqrt t
+\uc}\r]^N (1+\uc |y|)^{-N}
\end{eqnarray}
and
\begin{eqnarray}\label{3.3}
\hs\hs\hs U^{\ast}_{\uc,\,N}(x):=\sup_{\gfz{|y_1-x|<\sqrt{t}<\uc^{-1}}
{|y_2-x|<\sqrt{t}<\uc^{-1}}}\lf[\frac{\sqrt
t}{|y_1-y_2|}\r]^{\mu}|u(y_1,t)-u(y_2,t)|\lf[\frac{\sqrt t}{\sqrt t
+\uc}\r]^N (1+\uc |y_1|)^{-N},
\end{eqnarray}
where $\mu$ is as in \eqref{1.17} and $y,\,y_1,\,y_2\in\rn$.

\begin{lemma}\label{l3.1}
Let $L$ be an operator on $L^2(\rn)$ satisfying the assumptions $(A2)$ and $(A3)$,
and $\fai$ as in Definition \ref{d1.2}.  Then there exists a positive constant $C$,
depending on $n$ and $\fai$, such that,
for all $u$ as in \eqref{3.1}, $\uc\in(0,\fz)$ and $N\in\nn$,
\begin{eqnarray}\label{3.4}
\int_{\rn}\fai\lf(x,U^{\ast}_{\uc,\,N}(x)\r)\,dx\le C
\int_{\rn}\fai\lf(x,u^{\ast}_{\uc,\,N}(x)\r)\,dx,
\end{eqnarray}
where $u^{\ast}_{\uc,\,N}$ and $U^{\ast}_{\uc,\,N}$ are, respectively, as
in \eqref{3.2} and \eqref{3.3}.
\end{lemma}

\begin{proof}
For any $\az\in(0,\fz)$, measurable function $v:\ \rr^{n+1}_+\to\cc$ and $x\in\rn$, let
$$v^\ast_\az(x):=\sup_{|y-x|<\az\sqrt{t},\,t\in(0,\fz)}|v(y,t)|.
$$
Assume that $u$ is as in \eqref{3.1}. Fix $x\in\rn$. For any $y_1,\,y_2\in\rn$ and $t\in(0,\fz)$ satisfying
$|y_1-x|<\sqrt{t}$ and $|y_2-x|<\sqrt{t}$, let
$$v(y_1,t):=u(y_1,t)\lf[\frac{\sqrt{t}}{\sqrt{t}+\uc}\r]^{N}(1+\uc|y_1|)^{-N}\chi(\uc t),
$$
where $\chi$ denotes the characteristic function of $[0,1)$.
Then $v^\ast_1=u_{\uc,\,N}^\ast$. By the semigroup property of $\{e^{-tL}\}_{t>0}$,
we know that
\begin{eqnarray}\label{3.5}
\lf|u(y_1,t)-u(y_2,t)\r|=\lf|\int_{\rn}[K_{t/2}(y_1,z)-K_{t/2}(y_2,z)]u(z,t/2)\,dz\r|\le\mathrm{I}_0+\sum_{k=1}^\fz\mathrm{I}_k,
\end{eqnarray}
where
$$
\mathrm{I}_0:=\int_{B(y_1,\sqrt{t})}\lf|K_{t/2}(y_1,z)-K_{t/2}(y_2,z)\r|\lf|u(z,t/2)\r|\,dz$$
and, for each $k\in\nn$,
$$\mathrm{I}_k:=\int_{B(y_1,2^k\sqrt{t})\setminus B(y_1,2^{k-1}\sqrt{t})}\lf|K_{t/2}(y_1,z)-K_{t/2}(y_2,z)\r|\lf|u(z,t/2)\r|\,dz.$$
Furthermore, from \eqref{1.5}, \eqref{1.17} and the semigroup property of $\{e^{-tL}\}_{t>0}$, it follows that
\begin{eqnarray*}
\int_{B(y_1,2^k\sqrt{t})\setminus B(y_1,2^{k-1}\sqrt{t})}\lf|K_{t/2}(y_1,z)-K_{t/2}(y_2,z)\r|\,dz\ls\lf[\frac{|y_1-y_2|}{\sqrt{t}}\r]^{\mu}
e^{-\bz2^{2k}},
\end{eqnarray*}
where $\mu$ is as in \eqref{1.17} and $\bz$ is a positive constant determined by $c$ in \eqref{1.5},
which, combined with \eqref{3.5} and the fact that $(1+\uc|z|)^N\le(1+\uc|y_1|)^N(1+2^k)^N$ if $|y_1-z|<2^k\sqrt{t}$ and $\uc\sqrt{t}<1$,
further implies that
\begin{eqnarray*}
\lf[\frac{\sqrt{t}}{\sqrt{t}+\uc}\r]^N\frac{|u(y_1,t)-u(y_2,t)|}{(1+\uc|y_1|)^N}\ls
\lf[\frac{|y_1-y_2|}{\sqrt{t}}\r]^{\mu}\lf[v^\ast_4(x)+\sum_{k=1}^\fz e^{-\bz2^{2k}}(1+2^k)^N v^\ast_{2^{k+2}}(x)\r].
\end{eqnarray*}
By this, we conclude that
\begin{eqnarray}\label{3.6}
U^\ast_{\uc,\,N}(x)\ls v^\ast_4(x)+\sum_{k=1}^\fz e^{-\bz2^{2k}}(1+2^k)^N v^\ast_{2^{k+2}}(x).
\end{eqnarray}
Moreover, repeating the proof of \eqref{2.3}, we know that there exists $p_1\in(q(\fai),\fz)$ such that, for any $\az\in(1,\fz)$,
$$\int_{\rn}\fai\lf(x,v^\ast_\az(x)\r)\,dx\ls\az^{np_1}\int_\rn\fai\lf(x,v^\ast_1(x)\r)\,dx,
$$
where the implicit positive constant depends on $n$, $p_1$ and $\fai$,
which, together with \eqref{3.6}, Lemma \ref{l2.2}(i) and $v^\ast_1=u_{\uc,\,N}^\ast$, further implies that \eqref{3.4} holds true.
This finishes the proof of Lemma \ref{l3.1}.
\end{proof}

Now we prove Theorem \ref{t1.2} by using Lemma \ref{l3.1}.

\begin{proof}[Proof of Theorem \ref{t1.2}]
By the definitions of the spaces $H_{\fai,\,L,\,\mathrm{max}}(\rn)$ and $H_{\fai,\,L,\,\mathrm{rad}}(\rn)$
and the fact that $H_{\fai,\,L,\,\mathrm{max}}(\rn)\cap L^2(\rn)$ and $H_{\fai,\,L,\,\mathrm{rad}}(\rn)\cap L^2(\rn)$
are, respectively, dense in $H_{\fai,\,L,\,\mathrm{max}}(\rn)$ and $H_{\fai,\,L,\,\mathrm{rad}}(\rn)$,
to show Theorem \ref{t1.2}, it suffices to show that
\begin{eqnarray}\label{3.7}
\lf[H_{\fai,\,L,\,\mathrm{rad}}(\rn)\cap L^2(\rn)\r]\subset \lf[H_{\fai,\,L,\,\mathrm{max}}(\rn)\cap L^2(\rn)\r].
\end{eqnarray}

Let $f\in H_{\fai,\,L,\,\mathrm{rad}}(\rn)\cap
L^2 (\rn)$ and $u$ be as in \eqref{3.1}.  By \eqref{1.5}, we conclude that
$f^+_L\ls \cm(f)$, which, combined with the fact that, for all
$\uc\in(0,\fz)$ and $N\in\nn$, $u^{\ast}_{\uc,\,N}\ls f^+_L$,
implies that $u^{\ast}_{\uc,\,N}\ls \cm(f)$. From this and the boundedness of $\cm$ on $L^2(\rn)$,
we deduce that, for all $\uc\in(0,\fz)$ and $N\in\nn$, $u^{\ast}_{\uc,\,N}\in L^2(\rn)$.
Define
$$G_{\uc,\,N}:=\lf\{x\in\rn:\ U^{\ast}_{\uc,\,N}(x)\le Eu^{\ast}_{\uc,\,N}(x)\r\},$$
where $E$ is a positive constant determined later. By \eqref{3.4}, we know that
$$\int_{\rn}\fai\lf(x,U^{\ast}_{\uc,\,N}(x)\r)\,dx\le C
\int_{\rn}\fai\lf(x,u^{\ast}_{\uc,\,N}(x)\r)\,dx,$$
where $C$ is as in \eqref{3.4}. Let $p_0\in(0,i(\fai))$ be a uniformly lower type of $\fai$.
Take $E\in(1,\fz)$ large enough such that
$$\frac{C}{E^{p_0}}\int_{\rn}\fai\lf(x,u^{\ast}_{\uc,\,N}(x)\r)
\,dx\le\frac{1}{2} \int_{\rn}\fai\lf(x,u^{\ast}_{\uc,\,N}(x)\r)\,dx,$$
which, together with the definition of $G_{\uc,\,N}$ and the uniformly lower
type $p_0$ property of $\fai$ and the increasing property of $\fai$ about the variable $t$, implies that
\begin{eqnarray}\label{3.8}
\int_{\rn\setminus G_{\uc,\,N}}\fai\lf(x,u^{\ast}_{\uc,\,N}(x)\r)\,dx&&\le\int_{\rn\setminus
G_{\uc,\,N}}\fai\lf(x,\frac{U^{\ast}_{\uc,\,N}(x)}{E}\r)\,dx\\ \nonumber
&&\le\frac{C}{E^{p_0}}\int_{\rn}\fai\lf(x,u^{\ast}_{\uc,\,N}(x)\r)\,dx\\ \nonumber
&&\le\frac{1}{2}\int_{\rn}\fai\lf(x,u^{\ast}_{\uc,\,N}(x)\r)\,dx.
\end{eqnarray}
From \eqref{3.8}, it follows that
\begin{eqnarray}\label{3.9}
\int_{\rn}\fai\lf(x,u^{\ast}_{\uc,\,N}(x)\r)\,dx\le2\int_{G_{\uc,\,N}}
\fai\lf(x,u^{\ast}_{\uc,\,N}(x)\r)\,dx.
\end{eqnarray}
For all $x\in\rn$, let $M_r (f^+_L) (x):=\{M([f^+_L]^r)(x)\}^{1/r}$ with $r\in(0,1)$. Then,
for almost every $x\in G_{\uc,\,N}$, we have
\begin{eqnarray}\label{3.10}
u^{\ast}_{\uc,\,N}(x)\ls M_r (f^+_L)(x),
\end{eqnarray}
where the implicit positive constant depends on $n$, $\mu$ and $E$.
Indeed, let $x\in G_{\uc,\,N}$ such that $u^{\ast}_{\uc,\,N}(x)<\fz$. By
the definition of $u^{\ast}_{\uc,\,N}$, we know that there exist
$y\in\rn$ and $t\in (0,\fz)$ such that $|y-x|<\sqrt t<\uc^{-1}$ and
\begin{eqnarray}\label{3.11}
|u(y,t)|\lf[\frac{\sqrt t}{\sqrt t +\uc}\r]^N (1+\uc |y|)^{-N}\ge
\frac{1}{2}u^{\ast}_{\uc,\,N}(x).
\end{eqnarray}
Since $x\in G_{\uc,\,N}$, if $|z_1-x|<\sqrt t<\uc^{-1}$ and $|z_2-x|<\sqrt t<\uc^{-1}$,
then, from the definition of $U^{\ast}_{\uc,\,N}(x)$ and \eqref{3.11}, it follows that
\begin{eqnarray}\label{3.12}
&&\lf[\frac{\sqrt t}{|z_1 -z_2|}\r]^{\mu}|u(z_1,t)-u(z_2,t)|
\lf[\frac{\sqrt t}{\sqrt t+\uc}\r]^N (1+\uc|z_1|)^{-N}\\ \nonumber
&&\hs\le2E|u(y,t)|\lf[\frac{\sqrt t}{\sqrt t+\uc}\r]^N
(1+\uc|y|)^{-N}.
\end{eqnarray}
Take $E_t:=\{w\in\rn:\ |w-y|<\frac{\sqrt t}{2\wz{C}_1}\}$, where
$\wz{C}_1:=(4E)^{1/\mu}/2$. Obviously, $\wz{C}_1\ge1$.
Taking $z_1:=y$ and $z_2\in E_t$, by \eqref{3.12}, we find that
$$\lf[\frac{\sqrt t}{|y-z_2|}\r]^{\mu}|u(y,t)-u(z_2,t)|\le2E
|u(y,t)|.$$
From this and the choice of $\wz{C}_1$, it follows that
$|u(z_2,t)|\ge\frac{1}{2}|u(y,t)|$. Thus, we have
$$|u(z_2,t)|\ge\frac{1}{2}|u(y,t)|\ge\frac{1}{2}
|u(y,t)|\lf[\frac{\sqrt t}{\sqrt t+\uc}\r]^N (1+\uc|y|)^{-N}\ge
\frac{1}{4}u^{\ast}_{\uc,\,N}(x),$$
which further implies that
\begin{eqnarray*}
[M_r(f^+_L)(x)]^r&&\ge\frac{1}{|B(x,\sqrt t)|}\int_{B(x,\sqrt t)}[f^+_L(z)]^r \,dz\\ \nonumber
&&\ge\frac{1}{|B(x,\sqrt t)|}\int_{B(x,\sqrt t)}[u(z,t)]^r \,dz\\ \nonumber
&&\ge\lf[\frac{1}{4}u^{\ast}_{\uc,\,N}(x)\r]^r
\frac{|B_t|}{|B(x,\sqrt t)|}\sim\lf[\frac{1}{4}u^{\ast}_{\uc,\,N}(x)\r]^r.
\end{eqnarray*}
Thus, \eqref{3.10} holds true.

Let $\wz{q}\in(q(\fai),\fz)$, $p_0\in(0,i(\fai))$ and $r_0\in(0,1)$
such that $r_0\wz{q}<p_0$. Then $\fai$ is of uniformly
lower type $p_0$ and $\fai\in\aa_{\wz{q}}(\rn)$. For any
$\gz\in(0,\fz)$ and $g\in L^{\wz{q}}_{\loc}(\rn)$, let
$g=g\chi_{\{x\in\rn:\ |g(x)|\le\gz\}}+g\chi_{\{x\in\rn:\
|g(x)|>\gz\}}=:g_1+g_2$.
It is easy to see that
$$\lf\{x\in\rn:\ \cm(g)(x)>2\gz\r\}\subset\lf\{x\in\rn:\ \cm(g_2)(x)>\gz\r\},$$
which, combined with Lemma \ref{l2.2}(iii),
implies that, for all $t\in(0,\fz)$,
\begin{eqnarray}\label{3.13}
&&\int_{\{x\in\rn:\ \cm(g)(x)>2\gz\}}\fai(x,t)\,dx\\ \nonumber
&&\hs\le\int_{\{x\in\rn:\
\cm(g_2)(x)>\gz\}}\fai(x,t)\,dx\le\frac{1}{\gz^{\wz{q}}}\int_{\rn}
\lf[\cm(g_2)(x)\r]^{\wz{q}}\fai(x,t)\,dx\\ \nonumber
&&\hs\ls\frac{1}{\gz^{\wz{q}}}\int_{\rn}
|g_2(x)|^{\wz{q}}\fai(x,t)\,dx\sim\frac{1}{\gz^{\wz{q}}}\int_{\{x\in\rn:\
|g(x)|>\gz\}} |g(x)|^{\wz{q}}\fai(x,t)\,dx,
\end{eqnarray}
where the implicit positive constants depend on $n$, $\wz{q}$ and $\fai$.
From \eqref{3.13} and the definition of $M_{r_0}$, we deduce that, for any $\gz\in(0,\fz)$,
\begin{eqnarray}\label{3.14}
&&\int_{\{x\in\rn:\ M_{r_0}(f^+_L)(x)>\gz\}}\fai(x,t)\,dx\\ \nonumber
&&\hs\ls\frac{1}{\gz^{r_0\wz{q}}}\int_{\{x\in\rn:\ [f^+_L(x)]^{r_0}>\frac{\gz^{r_0}}{2}\}}
\lf[f^+_L(x)\r]^{r_0\wz{q}}\fai(x,t)\,dx\\ \nonumber
&&\hs\ls\sz_{f^+_L,\,t}\lf(\frac{\gz}{2^{1/r_0}}\r)+\frac{1}{\gz^{r_0\wz{q}}}
\int_{\frac{\gz}{2^{1/r_0}}}^{\fz}r_0\wz{q}s^{r_0\wz{q}-1}\sz_{f^+_L,\,t}(s)\,ds,
\end{eqnarray}
here and hereafter, $\sigma_{f^+_L,\,t}(\gamma):=\int_{\{x\in\rn:\
f^+_L(x)>\gz\}}\fai(x,t)\,dx$. Let
$$\mathrm{J}_{f^+_L}:=\int_{\rn}\fai\lf(x,f^+_L(x)\r)\,dx.$$
Then, by \eqref{3.9}, \eqref{3.10}, \eqref{3.14}, Lemma \ref{l2.2}(ii),
Fubini's theorem and the uniformly lower type $p_0$ property of $\fai$ and the fact that
$\fai$ is increasing for the variable $t$, we conclude that
\begin{eqnarray*}
&&\int_{\rn}\fai(x,u^{\ast}_{\uc,\,N}(x))\,dx\\
&&\hs\ls\int_{G_{\uc,\,N}}
\fai\lf(x, u^{\ast}_{\uc,\,N}(x)\r)\,dx\ls\int_{G_{\uc,\,N}}\fai\lf(x, M_{r_0}(f^+_L)(x)\r)\,dx\\
&&\hs\ls\int_{\rn}\fai\lf(x,M_{r_0} (f^+_L)(x)\r)\,dx\sim\int_{\rn}\int_0^{M_{r_0}(f^+_L)(x)}\frac{\fai(x,t)}{t}\,dt\,dx\\ \nonumber
&&\hs\ls\int_0^{\fz}\frac{1}{t}\int_{\{x\in\rn:\
f^+_L(x)>\frac{t}{2^{1/r_0}}\}}\fai(x,t)\,dx\,dt\\
&&\hs\hs+\int_0^{\fz}\frac{1}{t^{r_0\wz{q}+1}}\lf\{\int_{\frac{t}
{2^{1/r_0}}}^{\fz}r_0\wz{q}s^{r_0\wz{q}-1}\sz_{f^+_L,\,t}(s)\,ds\r\}\,dt\\
&&\hs\sim\mathrm{J}_{f^+_L} +\int_0^{\fz}r_0\wz{q}s^{r_0\wz{q}-1}\lf\{\int_0^{2^{1/r_0}s}
\frac{1}{t^{r_0\wz{q}+1}}\sz_{f^+_L,\,t}(s)\,dt\r\}\,ds\\
&&\hs\ls\mathrm{J}_{f^+_L}+\int_0^{\fz}r_0\wz{q}s^{r_0\wz{q}-1}
\sz_{f^+_L,\,t}(s)\fai\lf(x,2^{1/r_0}s\r)\lf\{\int_0^{2^{1/r_0}s}\lf[\frac{t}{2^{1/r_0}s}\r]^{p_0}
\frac{1}{t^{r_0\wz{q}+1}}\,dt\r\}\,ds\\
&&\hs\ls\mathrm{J}_{f^+_L}+\int_0^{\fz}r_0\wz{q}s^{r_0\wz{q}-1}\sz_{f^+_L,\,t}(s)
\frac{\fai(x,s)}{(2^{\frac{1}{r_0}}s)^{p_0}}
\lf\{\int_0^{2^{1/r_0}s}t^{p_0-r_0\wz{q}-1}\,dt\r\}\,ds\\
&&\hs\ls\mathrm{J}_{f^+_L} +\int_0^{\fz}\int_{\{x\in\rn:\
f^+_L(x)>s\}}\frac{\fai(x,s)}{s}\,dx\,ds\sim\int_{\rn}\fai\lf(x,f^+_L(x)\r)\,dx,
\end{eqnarray*}
where the implicit positive constants depend on $n$, $\mu$, $r_0$, $p_0$, $\wz{q}$ and $\fai$.
Letting $\uc\to0$, by the Fatou lemma, we have
\begin{eqnarray*}
\int_{\rn}\fai(x,f^{\ast}_L(x))\,dx\ls\int_{\rn}\fai(x, f^+_L(x))\,dx,
\end{eqnarray*}
which, together with the fact that, for any $\lz\in(0,\fz)$, $(f/\lz)^{+}_L=f^+_L/\lz$ and
$(f/\lz)^\ast_L=f^\ast_L/\lz$, implies that $$\int_{\rn}\fai\lf(x,\frac{f^\ast_L(x)}{\lz}\r)\,dx\ls\int_{\rn}\fai\lf(x,\frac{f^+_L(x)}{\lz}\r)\,dx.$$
From this, we deduce that
$\|f\|_{H_{\fai,\,L,\,\mathrm{max}}(\rn)}\ls\|f\|_{H_{\fai,\,L,\,\mathrm{rad}}(\rn)}$, which, combined with
the arbitrariness of $f\in H_{\fai,\,L,\,\mathrm{rad}}(\rn)\cap L^2
(\rn)$, further implies that \eqref{3.7} holds true. This finishes the proof of Theorem \ref{t1.2}.
\end{proof}

\bigskip

\noindent Dachun Yang

\medskip

\noindent School of Mathematical Sciences, Beijing Normal
University, Laboratory of Mathematics and Complex Systems, Ministry
of Education, Beijing 100875, People's Republic of China

\smallskip

\noindent{\it E-mail:} \texttt{dcyang@bnu.edu.cn}

\bigskip

\noindent Sibei Yang (Corresponding author)

\medskip

\noindent School of Mathematics and Statistics, Gansu Key Laboratory of Applied Mathematics and
Complex Systems, Lanzhou University, Lanzhou, Gansu 730000, People's Republic of China

\smallskip

\noindent{\it E-mail:} \texttt{yangsb@lzu.edu.cn}


\begin{thebibliography}{999}

\bibitem{adm} P. Auscher, X. T. Duong and A. McIntosh,
Boundedness of Banach space valued singular integral operators and
Hardy spaces, Unpublished Manuscript, 2005.

\vspace{-0.3cm}

\bibitem{ar03} P. Auscher and E. Russ, Hardy spaces and divergence operators on
strongly Lipschitz domains of $R^n$, J. Funct. Anal. 201 (2003), 148-184.

\vspace{-0.3cm}

\bibitem{at98} P. Auscher and Ph. Tchamitchian, Square root problem for
divergence operators and related topics, Ast\'erisque 249 (1998), viii+172\,pp.

\vspace{-0.3cm}

\bibitem{bckyy13b}
T. A. Bui, J. Cao, L. D. Ky, D. Yang and S. Yang,
Musielak-Orlicz-Hardy spaces associated with operators satisfying reinforced
off-diagonal estimates, Anal. Geom. Metr. Spaces 1 (2013), 69-129.

\vspace{-0.3cm}

\bibitem{cd12} F. Cacciafesta and P. D'Ancona, Weighted $L^p$ estimates for powers of
selfadjoint operators, Adv. Math. 229 (2012), 501-530.

\vspace{-0.3cm}

\bibitem{c77} A. Calder\'on, An atomic decomposition of distributions in
parabolic $H^p$ spaces, Adv. Math. 25 (1977), 216-225.

\vspace{-0.3cm}

\bibitem{cn95} D. Cruz-Uribe and C. J. Neugebauer, The structure
of the reverse H\"older classes, Trans. Amer. Math. Soc. 347
(1995), 2941-2960.

\vspace{-0.3cm}

\bibitem{dl13} X. T. Duong and J. Li, Hardy spaces associated to operators
satisfying Davies-Gaffney estimates and bounded holomorphic functional calculus,
J. Funct. Anal.  264  (2013), 1409-1437.

\vspace{-0.3cm}

\bibitem{dz02} J. Dziuba\'nski and J. Zienkiewicz, $H^p$ spaces for Schr\"odinger
operators, Fourier analysis and related topics, 45-53, Banach Center
Publ., 56, Polish Acad. Sci., Warsaw, 2002.

\vspace{-0.3cm}

\bibitem{fs72} C. Fefferman and E. M. Stein, $H^p$ spaces of several
variables, Acta Math. 129 (1972), 137-193.

\vspace{-0.3cm}

\bibitem{gra1} L. Grafakos,  Modern Fourier Analysis, Second edition,
Graduate Texts in Mathematics 250, Springer, New York, 2009.

\vspace{-0.3cm}

\bibitem{hlmmy} S. Hofmann, G. Lu, D. Mitrea, M. Mitrea and L. Yan,
Hardy spaces associated to non-negative self-adjoint operators
satisfying Davies-Gaffney estimates, Mem. Amer. Math. Soc. 214
(2011), no. 1007, vi+78 pp.

\vspace{-0.3cm}

\bibitem{hm09} S. Hofmann and S. Mayboroda,
Hardy and BMO spaces associated to divergence form elliptic
operators, Math. Ann. 344 (2009), 37-116.

\vspace{-0.3cm}

\bibitem{hyy} S. Hou, D. Yang and S. Yang, Lusin area function and molecular
characterizations of Musielak-Orlicz Hardy spaces and their applications,
Commun. Contemp. Math.  15  (2013),  no. 6, 1350029, 37 pp.

\vspace{-0.3cm}

\bibitem{jy10} R. Jiang and D. Yang, New Orlicz-Hardy spaces associated
 with divergence form elliptic operators, J. Funct.
 Anal. 258 (2010), 1167-1224.

\vspace{-0.3cm}

\bibitem{jy11} R. Jiang and D. Yang, Orlicz-Hardy spaces associated with operators
satisfying Davies-Gaffney estimates, Commun. Contemp. Math. 13 (2011), 331-373.

\vspace{-0.3cm}

\bibitem{jyy12} R. Jiang, Da. Yang and Do. Yang,
Maximal function characterizations of Hardy spaces associated with
magnetic Schr\"odinger operators, Forum Math. 24 (2012), 471-494.

\vspace{-0.3cm}

\bibitem{jn87} R. Johnson and C. J. Neugebauer, Homeomorphisms preserving $A_p$,
Rev. Mat. Iberoam. 3 (1987), 249-273.

\vspace{-0.3cm}

\bibitem{k} L. D. Ky, New Hardy spaces of Musielak-Orlicz type and
boundedness of sublinear operators, Integral Equations Operator Theory 78 (2014),
115-150.

\vspace{-0.3cm}

\bibitem{k1} L. D. Ky, Bilinear decompositions and commutators of singular
integral operators, Trans. Amer. Math. Soc. 365 (2013), 2931-2958.

\vspace{-0.3cm}

\bibitem{ls13} S. Liu and L. Song, An atomic decomposition of weighted Hardy spaces
associated to self-adjoint operators, J. Funct. Anal. 265 (2013), 2709-2723.

\vspace{-0.3cm}

\bibitem{m83} J. Musielak, Orlicz Spaces and Modular Spaces, Lecture Notes
in Math. 1034, Springer-Verlag, Berlin, 1983.

\vspace{-0.3cm}

\bibitem{rr91} M. Rao and Z. Ren, Theory of Orlicz Spaces,
Marcel Dekker, Inc., New York, 1991.

\vspace{-0.3cm}

\bibitem{sy10} L. Song and L. Yan, Riesz transforms associated to
Schr\"odinger operators on weighted Hardy spaces, J. Funct. Anal.  259  (2010),  1466-1490.

\vspace{-0.3cm}

\bibitem{sy15} L. Song and L. Yan, A maximal function characterization for Hardy spaces
associated to nonnegative self-adjoint operators satisfying Gaussian estimates,
Adv. Math. 287 (2016), 463-484.

\vspace{-0.3cm}

\bibitem{st93} E. M. Stein, Harmonic Analysis: Real-variable Methods,
Orthogonality, and Oscillatory integrals, Princeton Univ. Press,
Princeton, NJ, 1993.

\vspace{-0.3cm}

\bibitem{sw60} E. M. Stein and G. Weiss, On the theory of
harmonic functions of several variables. I. The theory of
$H^p$-spaces,  Acta Math. 103 (1960), 25-62.

\vspace{-0.3cm}

\bibitem{y08} L. Yan, Classes of Hardy spaces associated
with operators, duality theorem and applications, Trans. Amer.
Math. Soc. 360 (2008), 4383-4408.

\vspace{-0.3cm}

\bibitem{yy14} Da. Yang and Do. Yang, Maximal function characterizations of
Musielak-Orlicz-Hardy spaces associated with magnetic Schr\"odinger operators,
Front. Math. China 10 (2015), 1203-1232.

\vspace{-0.3cm}

\bibitem{yys2} D. Yang and S. Yang, Orlicz-Hardy spaces associated
with divergence operators on unbounded strongly Lipschitz domains of
$\mathbb{R}^n$, Indiana Univ. Math. J.  61 (2012), 81-129.

\vspace{-0.3cm}

\bibitem{yys4} D. Yang and S. Yang, Musielak-Orlicz-Hardy spaces associated with operators
and their applications, J. Geom. Anal.  24 (2014), 495-570.

\end{thebibliography}
\end{document}